\documentclass{amsart}
\usepackage{mathrsfs}
\usepackage[all]{xy}

\usepackage{wasysym}
\usepackage{amssymb}
\usepackage{comment}
\usepackage{mathtools}
\usepackage{mathbbol}
\input epsf
\usepackage{wrapfig}
\usepackage{lipsum}

\def\foral{\forall\hspace*{0.5mm}}
\def\exist{\exists\hspace*{0.7mm}}

\usepackage{ifpdf}
\ifpdf 
  \usepackage[pdftex]{graphicx}
  \DeclareGraphicsExtensions{.pdf,.png,.jpg,.jpeg,.mps}
  \usepackage{pgf}
\else 
  \usepackage{graphicx}
  \DeclareGraphicsExtensions{.eps,.bmp}
  \usepackage{pgf}
  \usepackage{pstricks}
\fi
\usepackage{epic,bez123}
\usepackage{wrapfig}

\DeclareMathAlphabet{\mathpzc}{OT1}{pzc}{m}{it}

\newtheorem{thm}{Theorem}[section]
\newtheorem{prop}[thm]{Proposition}
\newtheorem{cor}[thm]{Corollary}
\newtheorem{lem}[thm]{Lemma}

\newtheorem*{claim}{Claim}

\newtheorem{conv}[thm]{Convention}

\theoremstyle{definition}
\newtheorem{defn}[thm]{Definition}

\theoremstyle{remark}

\newtheorem*{rem}{Remark}
\newtheorem*{rems}{Remarks}

\newtheorem*{ack}{Acknowledgments}

\newcommand{\Gx }{\mathscr{G} (G, S)}
\newcommand{\GP }{(G, \mathcal P)}

\DeclarePairedDelimiter\floor{\lfloor}{\rfloor}

\newcommand{\fdo }{\rho_{\lambda, o}}
\newcommand{\flo }{\mathfrak l_{\lambda, o}}
\newcommand{\sfdo }{\bar{\rho}_{\lambda, o}}

\newcommand{\fds }{\rho}
\newcommand{\fls }{\mathfrak l}

\newcommand{\act}{\curvearrowright}

\newcommand{\Gf}{\overline{G}_\lambda}

\newcommand{\pGf}{\partial_\lambda{G}}
\newcommand{\cGf}{\partial_\lambda^c{G}}
\newcommand{\uGf}{\partial_\lambda^{uc}{G}}
\newcommand{\uGfLo}{\partial^{uc}_{L, o} G}
\newcommand{\uGfL}{\partial^{uc}_{L}G}
\newcommand{\pG}{\Lambda G}

\newcommand{\pT}{\partial{T}}
\newcommand{\pX}{\partial{X}}
\newcommand{\uG}{\Lambda^{uc}{G}}
\newcommand{\uGLo}{\Lambda^{uc}_{L, o} G}
\newcommand{\uGL}{\Lambda^{uc}_{L}G}

\newcommand{\PS}[2]{\Theta_{#1}(s,{#2})}
\newcommand{\dirac}[1]{{\mbox{Dirac}}{(#1)}}
\newcommand{\supp}{\mbox{supp}}
\newcommand{\isom}{\mbox{Isom}}

\newcommand{\g}[1]{\delta_{#1, S}}

\newcommand{\HD}{{\rm Hdim}}
\newcommand{\HDor}{\HD_{\bar \rho}}

\newcommand{\diam }[1]{{\|#1\|}}
\newcommand{\proj }{\textbf{Proj}}

\newcommand{\len }{\ell}

\usepackage[bookmarks=true, pdfauthor={YANG Wenyuan}]{hyperref}

\begin{document}

\title{Hausdorff dimension of boundaries of relatively hyperbolic groups}

\author{Leonid Potyagailo}
\address{ Leonid Potyagailo, UFR de Math\'ematiques, Universit\'e
de Lille 1, 59655 Villeneuve d'Ascq cedex, France}
\email{potyag@math.univ-lille1.fr}
\author{Wen-yuan Yang}
\address{Beijing International Center for Mathematical Research (BICMR), Beijing University, No. 5 Yiheyuan Road, Haidian District, Beijing, China}
\email{yabziz@gmail.com}

\thanks{This research is supported by the ERC starting grant GA 257110 ”RaWG” and by ANR grant
BLAN2011BS0101304.}



\date{April 11, 2016}

\dedicatory{}


\begin{abstract}
In this paper, we study the Hausdorff dimension of the Floyd  and Bowditch boundaries  of a relatively hyperbolic group, and show that for the Floyd metric and shortcut metrics respectively, they are are both equal to a constant times the growth rate of the group.

In the proof, we study a special class of conical points called uniformly conical points and establish that,  in both boundaries, there exists a sequence of Alhfors regular sets with dimension tending to the Hausdorff dimension and these sets consist of uniformly conical points.
\end{abstract}

\maketitle

\section{Introduction}
\subsection{Main results.}The main goal of the paper is to calculate the Hausdorff dimension of the limit set of a geometrically finite action
of a finitely generated group $G$ on a compactum $X.$ Every action $G\act X$ we consider is a {\it convergence} action, i.e. the induced action on the space of the distinct triples is discontinuous. We say that $G\act X$ is {\it minimal} if $X$ coincides with the limit set $\Lambda_XG$ (or $\Lambda G$ if $X$ is fixed) of the action, which is the set of the accumulation points of every orbit $Gx\ (x\in X).$

A point $\xi \in X$ is called \textit{conical}  if there exists a
sequence of elements $g_n \in G$ ($n\ge 1$)  such that the closure of $\{(g_n\xi, g_n\eta): n\ge 1\}$ in $X^2$ is disjoint from the diagonal $\Delta(X^2)=\{(x, x): x\in X\}$
for any $\eta \in X \setminus \xi$. If, in addition,  the set of elements $\{g_n g_{n+1}^{-1}: n\ge 1\}$ is in a uniformly bounded distance from the identity, then $\xi$ is called \textit{uniformly conical}. A quantitative version of an $L$-uniformly conical point for $L\ge 0$ is given in Definition \ref{mainconvdef}.

The action of a subgroup $H<G$ on $X$ is  {\it parabolic} if $H$ fixes a point $p\in X$, called {\it parabolic fixed point}. The parabolic action is {\it bounded parabolic} if $H$ acts properly and cocompactly on $X\setminus\{p\}.$ We will always assume that the action of the whole group $G$ is  non-parabolic so there is no a global fixed point.

A minimal non-parabolic action $G\act X$
is called \textit{geometrically finite} (or \textit{relatively hyperbolic}) if every point $x\in X$ is either conical or bounded parabolic (cf. Definition \ref{mainconvdef}). The stabilizer of a parabolic point is a {\it maximal parabolic} subgroup of $G$. We denote by $\mathcal P$ the set of maximal parabolic subgroups and call it {\it peripheral system} for the action. A group is called {\it relatively hyperbolic} with respect to    $\mathcal P$ if $G$ admits a geometrically finite action on $X$ with the  peripheral system  $\mathcal P$. If the compactum $X$ on which $G$ acts  is metrizable
then the action is geometrically finite if and only if the induced action on the space of distinct pairs is co-compact (we say in this case that the action on $X$  is {\it 2-cocompact}) \cite{Ge1}. If the opposite is not stated we will always assume that a relatively hyperbolic group is finitely generated and so $X$ is metrizable.

Let $G$
be a group with a finite generating set $S$. Assume that $1 \notin
S$ and $S = S^{-1}$. Consider the word metric $d_S$ on $G$. Denote
$B(n)=\{g \in G: d_S(1, g) \le n\}$ for $n \ge 0$. The
\textit{growth rate} $\delta_{G, S}$ of $G$ relative to $S$ is the
limit
$$
\g G = \lim\limits_{n \to \infty} \frac{\log \sharp B(n)}{n}.
$$

Recall that Floyd completion of a   group $G$ generated by $S$ is  the Cauchy completion of the Cayley graph $\Gx$  equipped with the distance $\rho_\lambda^o$  obtained by rescaling the length of an edge $e\in \Gx$ by a scalar function $\lambda^{d(e,o)}$ for a fixed $\lambda\in ]0,1[$ and a basepoint $o\in G$. The distance $\rho_\lambda^o$ is called Floyd distance at $o$,  and we use the notation  $\rho$ if $o$ and   $\lambda$ are clear from the context (see Subsection \ref{Section2.2} for more details). We   denote by $\overline G_\lambda$ and $\partial_\lambda G$ the corresponding Floyd completion and its boundary respectively.  By V.~Gerasimov's theorem  \cite[Proposition 3.4.6]{Ge2} for every finitely generated relatively hyperbolic group  the space $\pGf$ is the universal pullback space for every geometrically finite action of $G\act X$ in the sense that there exists an equivariant  continuous mapping $F:\pGf\to X$ (called Floyd map).

A.~Karlsson proved that the action of $G$ on the compact space $\Gf$ is a convergence action \cite{Ka}. Let  $\cGf$ (resp. $\uGf$) denote the
set of all (resp. uniformly) conical points for the action. We denote by $\HD_{\rho}$  the Hausdorff
dimension with respect to $\rho=\fdo$. The first main result of the paper is the following.

\begin{thm}\label{ThmHD}
Let $G$ be a relatively hyperbolic group with a finite generating
set $S$. There exists a constant $0< \lambda_0 <1$ such that
$$\HD_{\rho}(\pGf)=\HD_\rho(\cGf)=\HD_\rho(\uGf)=-\frac{\g G}{\log \lambda}$$ for any $\lambda \in [\lambda_0, 1)$.
\end{thm}
\begin{rem}
Note that for a hyperbolic group the Floyd metric is bilipschitz equivalent to the visual metric on the Gromov boundary (with appropriate choices of parameters). Even though this result seems to be a folklore, we have not found the corresponding reference in the literature. We provide a proof of it in the Appendix.  As a consequence the result of M. Coornaert \cite{Coor}  for the hyperbolic groups  is a partial case of Theorem \ref{ThmHD}.
\end{rem}

Note that the action of $G$ on the Floyd boundary $\pGf$ is not necessarily geometrically finite, as it is shown in \cite{YANG3} for Dunwoody's inaccessible groups.  In particular the Floyd boundary  is not in general homeomorphic to the limit set   $\pG$. So it is natural to ask if an analogous result to Theorem \ref{ThmHD} is true for $\pG$.


Consider  a minimal geometrically finite action of $G$ on a compact $X=\pG.$  It is shown in \cite{GePo2}  that the Floyd metric  $\rho$ transferred by the Floyd map  $F:\pGf\to\Lambda G$ is a metric on $\Lambda G$, called shortcut metric,  and is denoted by $\bar\rho$ (see subsection \ref{Section2.2}).

Our next goal is to  calculate
the Hausdorff dimension $\HD_{\bar\rho}$ of $\Lambda G$ with respect to $\bar\rho$. Denote by $\uG$ the set of uniformly conical points of $\pG$. The following theorem provides  the same conclusion  for the shortcut metric as in the case of the Floyd metric.

\begin{thm}\label{ThmHD2}
Let $G$ be a group with a finite generating
set $S$ acting geometrically finitely on a compactum $X=\pG$. Then there exists
a constant $0< \lambda_0 <1$ such that $$\HDor(\pG)=\HDor(\uG)=-\frac{\g G}{\log \lambda}$$
for any $\lambda \in [\lambda_0, 1)$.
\end{thm}

The above theorems implies the following.

\begin{cor}
For any shortcut metric $\bar \rho$, the Hausdorff dimension of the limit set of every relatively hyperbolic action of a group $G$ is constant and is equal to
$$\HD_\rho(\pGf)=\HDor(\pG)=-\frac{\g G}{\log \lambda}$$
for any $\lambda \in [\lambda_0, 1[$ where $\lambda_0 \in ]0, 1[$ is a fixed number.
\end{cor}

We say that a metric space $X$ is \textit{Ahlfors $Q$-regular} for a constant $Q>0$ if there exists a Borel measure $\mu$ on
$X$ such that the following holds
$$
\mu(B(x, r)) \asymp r^Q
$$
for any open ball $B(x, r)$ centered at $x \in X$ of radius
$r>0,$ where the symbol $\asymp$ denotes the bilipschitz equivalence between two quantities: $C^{-1} r^Q \leq \mu(B(x, r)) \leq C r^Q$  for a uniform constant  $C$.

Our next main result shows that the Hausdorff dimension of the Floyd boundary and of the limit set of a relatively hyperbolic action can be
well-approximated by a sequence of Ahlfors regular subsets.

\begin{thm}\label{ThmAR}
Let $G$ be a finitely generated relatively hyperbolic group with a finite generating
set $S$. Then there exists a sequence of Ahlfors $Q_i$-regular
subsets $X_i$ in $\pGf$ or $\pG$ such that $X_i$ consists of uniformly
conical points, $0 < Q_i < \g G$ and $Q_i \to \g G$ as $i \to \infty$.
\end{thm}

The proof of Theorem \ref{ThmAR} is based on the existence of an  $L$-transitional geodesic tree $\mathcal T=\mathcal T(L)\subset G$ (Lemma \ref{ITT}) depending on a parameter $L\gg 0$. Every vertex of $\mathcal T$ is a central point of  a geodesic interval, whose size depends on $L$,  and which belongs to a  neighbourhood of a left coset ({\it horosphere}) $gP$ where $P\in\mathcal P$  (see Subsection \ref{Section2.4}). We show that the endpoints of such a tree are $L$-uniformly conical (Lemma \ref{charuc}). However it is not true in general that every uniformly conical point appears as an endpoint of an $L$-transitional tree for a bounded $L$ (see the discussion after Lemma \ref{charuc}). The proof of Theorem \ref{ThmAR} shows that the Hausdorff dimension of the endpoints of $L$-transitional trees well approximate the Hausdorff dimension of the Floyd boundary (or the limit set) if $L\to\infty$. We recapitulate all these facts
in the following.

\begin{cor}\label{trantree} There exists a sequence $T_i$ of $L_i$-transitional trees such that $X_i=\partial T_i$ are Ahlfors $Q_i$-regular  spaces from the statement of Theorem \ref{ThmAR}.
\end{cor}

Our next result given in Section \ref{Section3}   generalizes the result of \cite{GePo3}  that the geodesics of the Floyd metrics are approximated by so called {\it tight paths}. Considering {\it generalized tight paths} (see Definition \ref{gentightdefn}) we show that they approximate the geodesics with respect to the  shortcut metric defined on the limit set.

Let $\uGfLo$ and $\uGLo$ denote subsets of uniformly conical points in $\uGf$ and $\uG$ depending on the above parameter $L$   (see Subsection \ref{Section2.4} for the precise definitions). The following result is central in this Section.

\begin{thm}[Proposition \ref{FloydMetric}]
Under the  assumptions of Theorem \ref{ThmHD} there exists $0< \lambda_0<1$ such that   for any $L>0$ and $\lambda\in [\lambda_0, 1[$ we have

  $$\fdo(\xi,
\eta) \asymp_{L} \lambda^{n},\; \forall \;\xi \ne \eta \in \uGfLo
$$
and
$$\sfdo(\xi,
\eta) \asymp_{L} \lambda^{n},\; \forall \;\xi \ne \eta \in \uGLo,$$ where $n=d(o, [\xi,\eta])$.
\end{thm}


\subsection{Historical remarks and motivations.} We provide here a short history of the study of the
Hausdorff dimension of the limit set of various convergence actions: Kleinian, hyperbolic and relatively hyperbolic.

The identification of  the  Hausdorff dimension with the critical exponent of Poincar\'e series was first established  by  S.~Patterson \cite{Patt}.  He introduced a probability measure on the limit set of  the  convex-cocompact Fuchsian groups,  and proved that up to a constant  it is equal to the Hausdorff measure. D.~Sullivan generalized this result and  constructed such measures (called since then   {\it Patterson-Sullivan measures}) on the limit sets of geometrically finite Kleinian groups  acting on the hyperbolic space ${\mathbb H}^n$ of dimension $n$ \cite{Sul}.   To finish the discussion of the case of Kleinian groups, we note the result of C.~Bishop and P.~Jones who proved in \cite{BishopJ} that for a non-elementary Kleinian group acting on the hyperbolic $3$-space  the Hausdorff dimension of the conical limit set is equal to the critical exponent of the Poincar\'e series (compare with our Theorems \ref{ThmHD} and \ref{ThmHD2}).    The latter results were generalized by F.~Paulin \cite{Paulin2} to discrete groups of isometries of Riemannian manifolds of strictly negative curvature.


M.~Coornaert has generalized the results of Patterson-Sullivan to the class of word-hyperbolic groups \cite{Coor}. In particular he proved that the Hausdorff dimension of the (Gromov) boundary of such a group with respect to the visual metric is equal to the critical exponent of the Poincar\'e series.

A natural question arises whether Coornaert's theorem holds for the class of relatively hyperbolic groups. However
 it was shown by M.~Burger and S.~Mozes   that if $G$ is a closed subgroup of the isometry group of a CAT($-1$) space $X$ and the parabolic subgroups of $G$ are not amenable then the critical exponent  is infinite  \cite[Proposition 1.6]{BuMo}. Such an  example
of a relatively hyperbolic group whose parabolic subgroups contain non-cyclic free subgroups  was constructed by D.~Gaboriau and F.~Paulin \cite[Example 1, p. 189]{GabPau}. By \cite{Paulin2} it then follows that  the Hausdorff dimension of the limit set
 for the action of such a group  with respect to the visual metric is infinite too.
  So in order to generalize Coornaert's theorem to the class of relatively hyperbolic groups one must replace the visual metric by a different one.

The Floyd metric obtained by a rescaling procedure of the word metric  is a natural candidat as it extends to the Floyd compactification of a group. Furthermore by a theorem of V.~Gerasimov   there exists  an equivariant and continuous map from the Floyd boundary $\partial_\lambda G$ to the limit set of any relatively hyperbolic action  of $G$ \cite{Ge2}. In particular, if $G$ is hyperbolic, the Floyd and Gromov boundaries are bilipschitz equivalent  for some exponential Floyd function.

M.~Bourdon has observed (private communication) that the Hausdorff dimension of the Floyd boundary of a relatively hyperbolic group, calculated with respect to the Floyd metric obtained with the exponential rescaling function $\lambda ^n\ (\lambda\in (0,1))$ is always upper bounded by $-\frac{\g G}{\log \lambda}$ (cf. Lemma \ref{Bourdon}). However the question whether it admits a lower strictly positive bound, which is equal to the same constant remained open. This was our first motivation giving rise to Theorem \ref{ThmHD}. Theorem \ref{ThmHD2} is then obtained by transferring the Floyd metric from the Floyd boundary $\partial_\lambda G$ to the limit set $\Lambda G$ of the geometrically finite action using the above  Gerasimov's map.

The lower bound estimate for the Hausdorff dimension in
Theorems \ref{ThmHD} and \ref{ThmHD2} follow from Theorem \ref{ThmAR}  providing the approximation of the boundary points by Ahlfors regular subsets  $X_i$. These subsets   entirely consist of uniformly conical points which are the space of ends of subtrees of the Cayley graph of $G.$ Note that the idea of such an approximation by trees is quite standard in  both settings: hyperbolic  (see e.g. \cite[6.1]{Gro}) or Kleinian (see \cite{BishopJ}).
However these constructions of trees  essentially use the hyperbolicity of the ambiant space. The latter property is not true for a relatively hyperbolic group: the Cayley graph is not in general hyperbolic and the relative Cayley graph is hyperbolic but the action on the set of vertices is not proper. The approximating trees constructed in the paper admits certain periodicity allowing us to obtain a   Patterson-Sullivan measure $\mu$ on $X_i$ also having periodic properties.   Theorem \ref{ThmAR} then shows that these measures converge to  the Hausdorff
measure  on a subset of uniformly conical points and whose dimension coincides with the full Hausdorff dimension of the ambiant space.

\ack The authors are deeply grateful to Marc Bourdon for several discussions
motivating our initial interest to the subject of the paper, we owe the proof of  Lemma \ref{Bourdon} to him
as well as  the references to the papers \cite{BuMo} and \cite{GabPau}. 

We also thank Misha Kapovich who indicated us that the bilipschitz equivalence between the Floyd and the visual metrics needs to be justified (see the Remark above).

During the work on this paper the first author was
partially supported by the ANR grant  DiscGroup (${\rm
BLAN}~2011\ BS01\ 013\ 04$).  He is also very grateful to the  Max-Planck-Institut f\"ur Mathematik in Bonn for the hospitality and support during his research stay at the Institute where a  part of the work   was done.

The
second author is grateful to   the CNRS    for providing him a
research fellowship   at the University of Lille 1, and was partially supported by the ERC starting grant GA 257110 RaWG when he was a postdoc in Orsay.

\section{Preliminaries}\label{Section2}

\subsection{Notations and Conventions}
Let $(Y, d)$ be a geodesic metric space. Given a subset $X$ and a
number $r \ge 0$, let $N_r(X) = \{y \in Y: d(y, X) \le r \}$. For $x\in Y$ denote  $B(x, r)=N_r(\{x\})$. Sometimes, we will write $B_d(x, r)$ to emphasize the metric $d$.


Given a point $y \in Y$ and a subset $X \subset Y$,
let $\proj_X(y)$ be the set of points $x$ in $X$ such that $d(y, x)
= d(y, X)$. The \textit{projection} of a subset
$A \subset Y$ to $X$ is then $\proj_X(A) = \cup_{a \in A} \proj_X(a)$.

We always consider a rectifiable path $\alpha$ in $Y$ with arc-length parametrization.  Denote by $\len (\alpha)$ the length
of $\alpha$, and by $\alpha_-$, $\alpha_+$ the initial and terminal points of $\alpha$ respectively.   Let $x, y \in \alpha$ be two points which are given by parametrization. Then denote by $[x,y]_\alpha$ the parametrized
subpath of $\alpha$ going from $x$ to $y$. We also denote by $[x, y]$ a choice of a geodesic in $Y$ between $x, y\in Y$.

A path $\alpha$ is called a \textit{$c$-quasi-geodesic} for $c\ge 1$ if the following holds
$$\len(\beta)\le c \cdot d(\beta_-, \beta_+)+c$$
for any rectifiable subpath $\beta$ of $\alpha$.

Let $\alpha, \beta$ be two paths in $Y$. Denote by $\alpha\cdot \beta$ (or simply $\alpha\beta$) the concatenated path provided that $\alpha_+ =
\beta_-$.

A path $\alpha$ going from $\alpha_-$ to $\alpha_+$ induces a first-last order as
follows. Given a property (P), a point $z$ on $\alpha$ is called
the \textit{first point} satisfying (P) if $z$ is among the points
$w$ on $\alpha$ with the property (P) such that $\len([\alpha_-, w]_\alpha)$ is
minimal. The \textit{last point} satisfying (P) is defined in a
similar way (replacing $[\alpha_-,w]_\alpha$  by $[w,\alpha_+]_\alpha$).

Let $f, g$ be real-valued functions with domain understood in
the context. Then $f \prec_{c_i} g$ means that
there is a constant $C >0$ depending on parameters $c_i$ such that
$f < Cg$,  and $\succ_{c_i}  $ is defined similarly. We use the symbol  $\asymp_{c_i}$ if both inequalities are true. For simplicity, we
omit $c_i$ if they are some universal constants.

Denote by $\diam\cdot$ the diameter of a set in a metric space. Recall the notion of Hausdorff measures in a metric space.
\begin{defn}
Let $X$ be a subset in a metric space $(Y, d)$. Given numbers
$\epsilon, s \ge 0$, define
$$
\mathcal H^s_{\epsilon}(X) = \inf\{ \sum \diam{U_i}^s: X \subset
\bigcup_{i=1}^{\infty} U_i, U_i \subset Y, \diam{U_i} \le \epsilon \}.
$$
Define $\mathcal H^s(X) = \lim\limits_{\epsilon \to 0} \mathcal H^s_{\epsilon}(X)$,
the \textit{$s$-dimensional Hausdorff measure} of $X$. The
\textit{Hausdorff dimension} of $X$ is defined as follows,
$$\HD_{d}(X) =
\inf\{s\ge 0: \mathcal H^s(X) =0 \} = \sup\{s\ge 0: \mathcal H^s(X) = \infty \}. $$ By
convention, set $\inf\emptyset = \sup\{s \in \mathbb R_{\ge 0}\}
= \infty$. Thus, $\HD_d{X} \in [0, \infty]$. Note that $\mathcal H^s(X)$ may
be zero for $s=\HD_d{X}$.
\end{defn}



\subsection{Floyd boundary and relative hyperbolicity}\label{Section2.2}
Let $G$ be a group with a finite generating set $S$. Assume that
$1\notin S$ and $S=S^{-1}$.  Let  $\Gx$ be  the \textit{Cayley graph}
of $G$ with respect to $S$. Denote by $d_S$ (or simply by $d$ if there is no ambiguity)
the word  metric on $\Gx$.

Fix $0 < \lambda <1$ and a basepoint $o \in G$. We define a Floyd metric $\fdo$ as follows. The \textit{Floyd
length} $\flo(e)$ of an edge $e$ in $\Gx $ is $\lambda^n$, where $n =
d(o, e)$. The Floyd length $\flo(\gamma)$ of a path $\gamma$ is the sum of Floyd lengths of its edges. This induces a length metric $\fdo$ on $\Gx$, which is the infimum of Floyd lengths of all possible paths between two points.

Let $\Gf$ be the Cauchy completion of $G$ with respect to $\fdo$.
The complement $\pGf$ of $\Gx$ in $\Gf$ is called \textit{Floyd
boundary} of $G$. The $\pGf$ is called \textit{non-trivial} if
$\sharp \pGf>2$. We refer the reader to \cite{Floyd},  \cite{Ge2}, \cite{GePo2}, \cite{Ka} for
more details.

By construction, the following equivariant property holds
\begin{equation}\label{equiv}
\rho_{\lambda, o} (x, y) =
\rho_{\lambda, go}(gx, gy)
\end{equation}
for any $g \in G$. The Floyd metrics with
different basepoints are related by a bi-Lipschitz inequality:
\begin{equation}\label{lambdabilip}
\lambda^{d(o, o')} \le \frac{\rho_{\lambda, o}(x, y)}{\rho_{\lambda, o'}(x, y)} \le
\lambda^{-d(o, o')}
\end{equation}
for any two points $o, o' \in G$.


We now recapitulate few standard definitions concerning geometrically finite convergence
 actions which will be often used further.
\begin{defn}\label{mainconvdef}
Let $X$ be a compact metrizable space on which $G$ admits a minimal and non-trivial convergence action  by
homeomorphisms.
\begin{enumerate}
\item
A point $\xi \in X$ is called \textit{conical}  if there exists a
sequence of elements $g_n \in G$ ($n\ge 1$)  such that the closure of $\{g_n(\xi, \eta): n\ge 1\}$ in $X^2$ is disjoint from the diagonal $\Delta(X^2)=\{(x, x): x\in X\}$
for any $\eta \in X \setminus \xi$.

If, in addition, there exists $L>0$ such that $d(1, g_n g_{n+1}^{-1}) \le L$, then $\xi$ is called \textit{$L$-uniformly conical} (or \textit{uniformly conical} if the constant $L$ is not important).
\item
A point $\xi \in X$ is called \textit{bounded parabolic} if the
stabilizer $G_\xi$ of $\xi$ in $G$ is infinite, and acts properly
and co-compactly on $X\setminus \xi$. The subgroup $G_\xi$ is called  {\it maximal parabolic}.
\item
A convergence group action of $G$ on $X$ is called \textit{geometrically
finite} if every limit point $\xi \in X$ is either a conical point or a
bounded parabolic point.
\end{enumerate}

\end{defn}

As it was mentioned in the Introduction
  a pair   $\GP$ is  \textit{relatively hyperbolic} if $G$ admits a geometrically finite group action on a compact metrizable space $X$ such that $\mathcal P$ coincides with the collection of maximal parabolic subgroups (peripheral system). Using the relative Cayley graph one can construct the limit set $\Lambda G$ of the action with the boundary of this graph \cite{Bow1}. We will often  call \textit{Bowditch boundary} the limit set $\pG$ of a geometrically finite action.  Bowditch proved that if $G$ is finitely generated then $\pG$ up to an equivariant homeomorphism depends only on the pair $\GP$ \cite{Bow2}. We also note the same result still holds in general case when $G$ is not finitely generated \cite[Corollary 6.1.e]{GePo5}.

The following result establishes the following universal \textit{pullback} property of the  Floyd boundary.

\begin{prop} \cite{Ge2}\label{Floydmap}
Suppose $\GP$ is a relatively hyperbolic pair. Then there exists $0< \lambda_0 <1$ such that for any  $\lambda \in [\lambda_0, 1)$ there exists a
continuous $G$-equivariant surjective map (called {\it Floyd map}):
$$
F_\lambda: \pGf \to \pG.
$$
\end{prop}

\medskip

Let $G_p$ be the stabilizer of a parabolic point $p\in X$ for the action $G\act X=\Lambda G$. Denote by $\Lambda_{\pGf}(G_p)$  and $\partial_\lambda G_p$ the limit set of $G_p$ for its action   on the Floyd boundary $\pGf$ of $G$, and the Floyd boundary of $G_p$ respectively.  The following result precisely  describes the kernel of the Floyd map.

\begin{prop} \cite{GePo2} \label{FloydKernel}
Under the assumption of Proposition \ref{Floydmap}, the following holds
$$
F_\lambda^{-1}(p) = \Lambda_{\pGf}(G_p)=\partial_\lambda G_p
$$
for any parabolic point $p$ in $\pG$. Moreover, $F_\lambda^{-1}(p)$ consists of one point if $p$ is a conical point.
\end{prop}

We   equip $\pG$ with a \textit{shortcut metric} as follows: let
$$\omega=\{(\eta, \xi) \in \pGf \times \pGf: F_\lambda(\xi)=F_\lambda(\eta)\}$$
be the relation on $\pGf$ given by the Floyd map $F_\lambda: \pGf \to \pG$. For any $\xi, \eta \in \Gf$, define a pseudo-distance $\tilde \rho_{\lambda, o}(\xi, \eta)$ on $\Gf$ to be
\begin{equation}\label{shortcutdefn}
\tilde \rho_{\lambda, o}(\xi, \eta)=\inf_{n\ge 1} \{\sum_{i=1}^{n} \fdo(\xi_i,\eta_{i}): (\eta_i,\xi_{i+1}) \in \omega, 1\le i < n, \xi_1=\xi,\; \eta_{n}=\eta\}.
\end{equation}

We have
\begin{equation}
\label{floydshortcut}
\forall \xi, \eta\in\Gf\  :\  \tilde \rho_{\lambda, o}(\xi, \eta)\leq   \fdo(\xi, \eta),\end{equation}
and  it is a maximal pseudo-metric on $\Gf\times\Gf$ satisfying this inequality. It is shown in \cite{Ge2} that the space $\widetilde{\Lambda G}:=\Lambda G\sqcup\Gx $ (called {\it attractor sum})     is compact. The action  $G\act\widetilde{\Lambda G} $ is convergence such that its restriction  on $\Gx$ is the identity and on $ \Lambda G$ it coincides with the initial action. Furthermore the Floyd map $F_\lambda$ extends to an equivariant continuous map (denoted by the same symbol):
$$F_\lambda : \Gf\to \widetilde{\Lambda G}$$ such that $F_\lambda\vert_G\equiv{\rm id}.$
Pushing forward $\tilde \rho_{\lambda, o}$ with $F_\lambda$, we obtain a \textit{shortcut
pseudo-metric} on $\widetilde{\Lambda G}:$
\begin{equation}\label{shortcut}
 \forall
x,y\in \widetilde{\Lambda G}:\; \sfdo(x, y)=\tilde \rho_{\lambda, o}(F_\lambda^{-1}(x), F_\lambda^{-1}(y)),
\end{equation}
 which turns out to be a real metric on $\widetilde{\Lambda G}$ (see \cite[Section 3]{GePo2} for details).
By the above construction, one can easily see that the shortcut metrics $\sfdo$ satisfy the properties (\ref{equiv}) and (\ref{lambdabilip}) too.

\begin{conv}\label{constlambda}
Since now on we will always suppose that $\lambda \in [\lambda_0, 1)$ where $\lambda_0$ is given by Proposition \ref{Floydmap}. We omit the index $\lambda$ in $\flo, \fdo$ and $\sfdo$ if $\lambda$ is given in the context.
\end{conv}

Finally, we recall the following Visibility Lemma.
\begin{lem}[Visibility lemma] \label{karlssonlem}\cite{Ka}
There is a function $\varphi: \mathbb R_{\ge 0} \to \mathbb R_{\ge 0}$ such that for
any $v \in G$ and any geodesic $\gamma$ in $\Gx$, we have if
$\fls_v(\gamma) \ge \kappa,$ then $d(v,
\gamma) \le \varphi(\kappa)$.
\end{lem}
\begin{rem}
The same result is valid for quasi-geodesics or more general $\Theta$-geodesics where $\Theta:\Bbb N\to G$ is a polynomial distortion function \cite[Lemma 5.1]{GePo2}.
\end{rem}

\subsection{Floyd geodesics}\label{Section2.3}
In this subsection, we provide a few basic tools to study Floyd geodesics.

We say that a path $\alpha: \Bbb Z\to \Gx$ \textit{ends at} $\xi \in \pGf$ if $\xi=\lim_{n\to\infty}\alpha(n)$. Denote in this case $\alpha_+=\xi$, and $\alpha_-=\lim_{n\to-\infty}\alpha(n)$. It follows from Lemma \ref{karlssonlem} that every geodesic ray ends at a point of the Floyd boundary. Moreover,  $\Gf$ is a geodesic metric space and is a visual boundary: any two distinct points $\xi, \eta \in \Gf$ are connected by a bi-infinite word geodesic belonging to the Cayley graph \cite[Proposition 2.4]{GePo2}.

We note that a Floyd geodesic between $\xi, \eta$ does not necessarily belong to the graph (e.g. an example of such situation is given by the Floyd geodesic $[n, +\infty]\cup [-\infty, -n]$ between $-n$ and $n$ for the group $\mathbb Z+\mathbb Z)$. A method to  overcome this problem was proposed in \cite{GePo3}. It consists in  introducing  a special type of paths called {\it tight paths} (see Definition \ref{tightdefn} in section \ref{Section3}) situating in the Cayley graph which will  approximate well the Floyd geodesics.  To provide a certain development of  this method  we will  need the following preliminary statements.

\begin{lem}\label{Floydgeod}\cite[Lemma 7.2]{GePo3}
For any $l>0$, there exists $0< \lambda_0 < 1$ such that the following property hold for any $\lambda\in [\lambda_0, 1)$.

Let $x, y \in \Gx$ such that $d(x, y) \le l$, and $p$ be a path with $\alpha_-=x$ such that $\len(\alpha) \ge d(x, y)+1$. Then $\flo(\alpha) > \fdo(x,y)$. In particular, the $\fdo$-geodesic between $x, y$ is a geodesic in $\Gx$
\end{lem}

We consider the following \textbf{shortening procedure} introduced in \cite{GePo3}: consider two points $x, y \in \Gf$, we take a sequence of paths $\gamma_n$ in $\Gx$ such that  $(\gamma_n)_-\to x, \;(\gamma_n)_+\to y$ and
$$
\flo(\gamma_n) \to \fdo(x, y).
$$
For every $l>0$ we can choose $\lambda_0 \in ]0, 1[$ such that $\gamma_n$ is an $l$-local geodesic. Indeed,  if a segment between two points of $\gamma_n$ at  distance at most $l$ is not a geodesic, then it can be replaced by a geodesic.  Applying  this  procedure several times, we obtain a $l$-local geodesic, still denoted by $\gamma_n$, whose Floyd length is not increased by Lemma \ref{Floydgeod} (see Lemma \ref{tithapprox} for more details).

The following lemma  states that word geodesic rays are also  Floyd and shortcut geodesics.

\begin{lem}\label{floydrays}
Let $o\in G$ be a base point and  $\gamma$ be a geodesic ray with $\gamma_-=o$. Then for any $v\in \gamma$ we have  $$\flo([v, x]_\gamma)= \sfdo(v, y),$$
 and
$$\flo([v, x]_\gamma)=\fdo(v, x),$$
where $x=\gamma_+ \in \pGf$ and $y=F(x)\in \pG$ where $F$ is the Floyd map given in Proposition \ref{Floydmap}.
\end{lem}
\begin{proof}
We only prove the result for the shortcut metric. In the case of Floyd metric   a straightforward calculation shows that  the geodesic ray   $\gamma$  as well as every its subray is also a $\fdo$-Floyd geodesic.

By definition (\ref{shortcutdefn}) of a shortcut metric, for any $n \in \mathbb N$, there exist
 pairs $(\eta_i, \xi_{i+1}) \in \omega$ where $1\le i < m$ such that
$$
\sfdo(v, y) \ge \sum_{1\le i \le m} \fdo (\xi_i, \eta_{i}) -\frac{1}{2n},
$$
where $\xi_1=v, \;\eta_{m}=x$.
 Every geodesic ray $[o, \eta_1]$ is also a Floyd geodesic so we can choose $\tilde \eta_1\in [o, \eta_1]$ such that $\fdo(\tilde \eta_1, \eta_1)\le \frac{1}{2n}$.   It follows $$\sfdo(v, y)\ge \fdo(v,\tilde\eta_1)-{1\over {n}}.$$

Choose $w\in [v, y]_\gamma$ such that $d(v, w)=d(v, \tilde \eta_1)=m$. Then the following is true: \begin{equation}\label{techn}\fdo(v, \tilde\eta_1)\geq\fdo(v,w).\end{equation}

Indeed, connect $v$ and $\tilde\eta_1$ by a curve $\alpha$. There exists a point $u=\alpha(t_0)$ such that $d(v,u)=m $ and choose a sub-curve $\alpha'=[v,u]_\alpha$ containing $m$ edges.  Since $\gamma$ is a word geodesic,  for the k-th edge $e\in \alpha'$ and the $k$-th edge $e_1\in[v,w]_\gamma$ we have $\flo (e)\geq \flo (e_1)\ (k\in\{0,...,m\}).$ Then $\flo(\alpha)\geq\flo(\alpha')\geq\fdo(v,w).$ So  (\ref{techn}) follows.

 \medskip


 \centerline{\includegraphics[width=7cm, height=4.5cm]{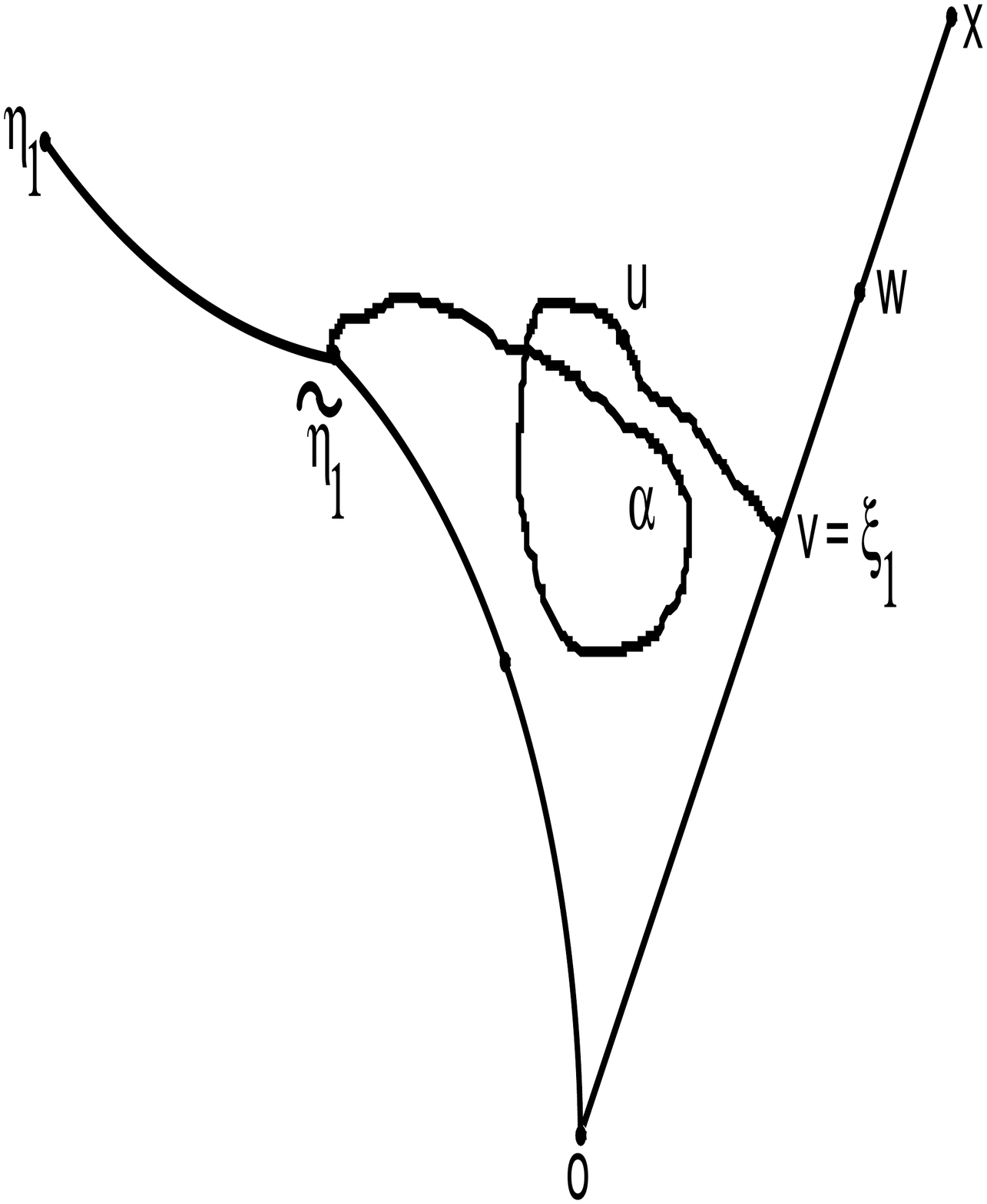}}
 \medskip

We have

$$\sfdo(v,y)\geq\fdo(v,\tilde\eta_1)-{1\over 2n}\geq \fdo(v,w)-{1\over n}\geq$$
$$\geq \flo([v,x]_{\gamma})-{\lambda^n\over {1-\lambda}}-{1\over n}.$$

 Passing to the limit    we obtain $\sfdo(v,y)\geq \flo([v, x]_\gamma)=\fdo(v, x)$. Since  $\fdo(v, x)\ge \sfdo(v, y)$, we conclude that $\flo([v, \gamma_+]_\gamma)= \sfdo(v, \gamma_+)$.
\end{proof}

\subsection{Transitional paths and uniformly conical points}\label{Section2.4}
In this subsection we shall give a description of
uniformly conical points in $\pG$ using the geometry of Cayley graph.

Let $\GP$ be a relatively hyperbolic pair. Denote $\mathbb
P=\{gP: g \in G, P\in \mathcal{\tilde P}\}$, where $\mathcal{\tilde P}$ is a maximal set of non-conjugate subgroups in $\mathcal P$. Following \cite{GePo2} we call  the elements of $\mathbb P$ {\it horospheres}.

\begin{defn}
Fix $\epsilon, R>0$. Let $\gamma$ be a path in $\Gx$ and $v \in
\gamma$ a vertex. 
Given $X \in \mathbb P$, we say that $v$ is \textit{$(\epsilon,
R)$-deep} in $X$ if $$\gamma \cap B(v, R) \subset
N_\epsilon(X).$$
If $v$ is not $(\epsilon, R)$-deep in any $X \in \mathbb P$, then
$v$ is called an \textit{$(\epsilon, R)$-transition point} of
$\gamma$.
\end{defn}

The following lemma together with Lemma \ref{karlssonlem} will be invoked several times.
\begin{lem}\label{floydtrans}
\begin{enumerate}

\item
For any $c\ge 1, R> 0$, there exists $\epsilon=\epsilon(c), \kappa =
\kappa(\epsilon, R)>0$ such that for any $c$-quasi-geodesic $\gamma$
and an $(\epsilon, R)$-transitional point $v$ in $\gamma$, we have
$$
\fds_v(\gamma_-, \gamma_+)  \ge \bar \fds_v(\gamma_-, \gamma_+) > \kappa.
$$

\item  For any $c\ge 1, \kappa, \epsilon>0$ there exists
$R=R(c, \kappa, \epsilon)>0$ such that for any $c$-quasi-geodesic $\gamma$
and a point $v\in \gamma$ with $\bar\fds_v(\gamma_-, \gamma_+)>\kappa$,  we have that $v$ is an $(\epsilon, R)$-transitional  point of $\gamma$.

\end{enumerate}
\end{lem}

\begin{proof}
Let us first prove (2). Suppose not, then $\exists c\ge 1, \kappa, \epsilon >0 : \forall n,  \exists$ $c$-quasi-geodesics $\gamma_n$ and $v_n\in\gamma_n$ such that $v_n$ is $(\epsilon, n)$-deep and $\bar\fds_{v_n}((\gamma_n)_-, (\gamma_n)_+)>\kappa$. Up to a normalization we may assume that $v_n=v=\gamma_n(0)$. Then $\gamma_n(]-n,n[)\subset N_{\epsilon}( X_n)$ for $X_n\in \mathbb P$. By compactness of geodesics  in the Tikhonoff topology, we obtain a limit horocycle $\alpha$ such that $\alpha_{\pm}=q$
and every part of $\alpha$ belongs to $\gamma_n$ for sufficiently large $n$ (see \cite[Prop. 5.2.3]{GePo3} for more details). Then the diameter of $ \partial(\gamma_n\cap\alpha) $ with respect to the distance $\overline\rho_v$ tends to $0$. As $\gamma_n$ are geodesics whose all interior points are in the graph we must have $\bar\fds_v((\gamma_n)_-, (\gamma_n)_+)\to 0$ which is a contradiction.

(1). By \cite[Corollary 3.9]{GePo2} there exists a constant $\epsilon=\epsilon(c)$  such that for every $X\in\mathbb P$ any  $c$-quasi-geodesic with endpoints in $X$ lies in $N_\epsilon(X)$ (all horospheres are uniformly quasi-convex). For the constants $c$ and $\epsilon=\epsilon(c)$  the statement now follows from \cite[Corollary 5.10]{GePo4}, following a similar argument as above.
\end{proof}

We introduce a special class of paths, which plays an important role
in the present study.
\begin{defn}\label{trans}
 Given $\epsilon, R, L>0$, a path $\gamma$ in $\Gx$ is called \textit{$(\epsilon, R, L)$-transitional} (or simply \textit{transitional} if the choice of the constants is not important) if for any
point $v \in \gamma$, there exists an $(\epsilon, R)$-transitional
point $w\in \gamma$ such that $\len([v, w]_\gamma) \le L$.

We say that an infinite path $\gamma$ in $\Gx$ is \textit{eventually} $(\epsilon, R, L)$-transitional if there exists $v\in \gamma$ such that $[v, \gamma_+)_\gamma$ is
$(\epsilon, R, L)$-transitional.
\end{defn}

We fix the constant $\epsilon=\epsilon(1) >0$ given by Lemma \ref{floydtrans}.1. The following lemma characterizes uniformly conical points as the endpoints of transitional geodesic rays.

\begin{lem}\label{charuc}
Let $\GP$ be a relatively hyperbolic pair.  There exists $R>0$ for which tplhe following property is true:

a point $\xi \in \pG$ is uniformly conical if and only if some (or any) geodesic ray ending at $\xi$ is eventually an $(\epsilon, R, L)$-transitional geodesic ray for some $L>0$.
\end{lem}

{\it Proof of} $``\Rightarrow".$
Since $G$ acts geometrically finitely on $\pG$,  it follows from \cite[Theorem 1C]{Tukia}  that there exists $\delta>0$ such that for any conical point $\xi\in \pG$, there exists a sequence $(g_n)\subset G$ such that for all points $\eta\in(\pG\cup G)\setminus\xi$ one has   $\bar\fds_1(g_n\xi, g_n\eta) > \delta$. Denote $r_0 :=
\varphi(\delta/2)$, where $\varphi$ is given by Lemma \ref{karlssonlem}.

  Assume that $\xi$ is an $L$-uniformly conical point for some $L>0$. Let $\gamma=[\gamma_-, \xi[$ be a   geodesic ray ending at $\xi$ and $(g_n)\subset G$ be the above sequence taken for the pair $(\gamma_-,\xi).$  Then  $\bar\fds_1(g_n\xi, g_n\gamma_-)=\bar\fds_{g_n^{-1}}(\xi, \gamma_-) > \delta/2$ and $d(1, g_n g_{n+1}^{-1}) \le L$ for all $n \ge 1$.  By Lemma \ref{karlssonlem}, $\gamma \cap B(g_n^{-1}, r_0) \neq \emptyset$ for $n\ge 1$. Let $v_n \in \gamma$ such that $d(v_n,
g_n^{-1})<r_0$. By the inequality (\ref{lambdabilip})  such that $\bar\fds_{v_n}(\gamma_-, \xi) > \kappa$ where $\kappa=\lambda^{r_0}\cdot\delta/2$ is a uniform constant. Moreover, $d(v_n, v_{n+1}) \le L + 2r_0$.

Hence, Lemma \ref{floydtrans}.2 gives rise to a uniform constant $R$ for which $v_n$ are all $(\epsilon, R)$-transitional for $n\ge 1$.


{\it Proof of} $``\Leftarrow".$
Let $\gamma$ be an $(\epsilon, R, L)$-transitional geodesic ray at $\xi=\gamma_+$, and $v_n$ ($n \ge 0$)  a sequence of $(\epsilon,
R)$-transitional points in $\gamma$ such that $d(v_n, v_{n+1}) \le L$ and $v_n \to \xi$. Then $\bar\fds_{v_n}(\gamma_-, \xi)\ge \kappa$, where $\kappa>0$ is given by Lemma \ref{floydtrans}.1. Denote $g_n:=v_n^{-1}$. Then $\bar\fds_1(g_n \gamma_-, g_n\xi) \ge \kappa$. In other words, $\{(g_n \gamma_-, g_n\xi)\}$ lies outside a uniform neighborhood of the diagonal $\Delta(\pG^2)$.

Since the action is convergence the point $\xi$ is conical. As $d(1, g_ng_{n+1}^{-1}) \le L$ it is uniformly conical.
\qed

\begin{rems}\label{sizef}
\begin{enumerate}
\item
The proof of the $``\Leftarrow"$ direction equally applies to a conical point in Floyd boundary $\pGf$ without assuming the geometrical finiteness of the action.

\item
The existence of the uniform constant $\delta>0$ which measures the size of a compact fundamental set for the co-compact action of $G$ on the set of distinct pairs   was only used to prove the implication $``\Rightarrow"$ (in order to get a uniform constant $R$).  The existence of such a constant implies that the  action of $G$  on a metrizable space $\Lambda G$ is 2-cocompact; the converse statement that a    2-cocompact and  non-elementary convergence action is geometrically finite is shown in \cite{Ge1}, and its proof does not request the metrisability of the space $X=\Lambda G.$
\item
As a corollary we see that for each $L>0$ the set of $L$-uniformly conical points is $G$-invariant, although this is not clear at all from the dynamical definition.
\end{enumerate}
 \end{rems}

\begin{cor}
Let $\epsilon=\epsilon(1) >0$ given by Lemma \ref{floydtrans}.1.
For any $R, L>0$, an  $(\epsilon, R, L)$-transitional geodesic ray ends at a uniformly conical point $\xi \in \pGf$.
\end{cor}

As another consequence of the proof, we have the following result.

\begin{cor}\label{ballcovering}
Let $G\act X$ be a geometrically finite action. Then there exists a constant $L >0$ such that for any conical point  $\xi \in X$  there is a sequence of elements $g_n \in G$ such that for any
geodesic $\gamma$ ending at $\xi$, we have
$$[v, \xi[_\gamma \subset \cup_{n\ge 1} B(g_n, L)$$
for some $v \in \gamma$. 
\end{cor}

\begin{rem}
In the setting of Kleinian groups, this property is used to define uniformly conical points, cf. \cite{Strat}. Here we do not need to assume that $G$ acts geometrically finitely on $\pGf$. Also the corollary holds for ``quasi-geodesics" instead of ``geodesics".
\end{rem}

We setup some notations for future discussions about uniformly conical points.

Let $\epsilon, R$ be given by Lemma \ref{charuc}. Denote by $\uGL$ the set of uniformly conical points $\xi \in \pG$ such that there exists an $(\epsilon, R, L)$-transitional geodesic ray $\gamma$ ending at $\xi$. It is obvious that $\uGL$ is a $G$-invariant set.

Fixing a basepoint $o\in G$, denote by $\uGLo$ the set of all uniformly conical points $\xi \in \uGL$ where
a geodesic $\gamma$ between $o$ and $\xi$ is $(\epsilon, R, L)$-transitional.

Clearly, $G\cdot \uGLo =\uGL$.
Thus, the set $\uGLo$ can be thought as a
fundamental domain for the action of $G$ on the set $\uGL$.

Similarly, we define the set of uniformly conical points $\uGfLo$ and $\uGfL$ on the Floyd boundary $\pGf$.  By Proposition \ref{FloydKernel}, there exists one-to-one correspondence between $\uGL$ and $\uGfL$.

\subsection{Contracting property}
Recall that $\diam{\cdot}$ denotes the diameter of a set in a metric space.
\begin{defn}\label{contractdefn}
For $c\ge 1$, a subset $X$ is called \textit{$c$-contracting} in a metric space $Y$ if there exists $\mu_c, D_c>0$ such that the following holds
\begin{equation}\label{contractingproperty}
\diam{\proj_X(\gamma)} < D_c
\end{equation}
for any $c$-quasi-geodesic $\gamma$ in $Y$ with $N_{\mu_c}(X) \cap \gamma =
\emptyset$.

A collection of $c$-contracting subsets is referred to as a
$c$-\textit{contracting system} if $\mu_c, D_c$ depends only on $c$.
\end{defn}

A system $\mathbb X$ has \textit{a bounded intersection property} if for any $\epsilon >0$
there exists $\mathcal R =\mathcal R(\epsilon)>0$ such that
$$
\diam{N_\epsilon(X) \cap N_\epsilon(X')} <\mathcal R$$ for any two distinct
$X, X' \in \mathbb X$.

In what follows, our discussion applies to the Cayley graph of a relatively hyperbolic group $\GP$ with a finite generating set $S$. In particular, we are interested in the contracting system with bounded intersection given by the following lemma.
\begin{lem}\label{periphsystem}\cite{GePo2}
Let $\mathbb P=\{gP: g \in G, P\in \mathcal{\tilde P}\}$, where $\mathcal{\tilde P}$ is a complete set of conjugacy representatives in $\mathcal P$. There exists $\mathcal R: \mathbb R_{>0} \to \mathbb R_{>0}$ such that the collection $\mathbb
P$ is a $c$-contracting system with the $\mathcal R$-bounded intersection for each $c\ge 1$.
\end{lem}
\begin{proof}
The contracting property is proven in \cite[Proposition 8.5]{GePo2}, and the bounded intersection is in \cite[Corollary 5.7]{GePo2}.
\end{proof}

The following lemma will be  often used further on.

\begin{lem}\label{uniformtrans}
Let $\mathbb P$ be the collection of horospheres in (\ref{periphsystem}). For any $c\ge 1$,  there exist $\epsilon_c=\epsilon(c)>0$ such that  for every $c$-quasi-geodesic $\gamma$ in $\Gx$ and $\epsilon \ge
\epsilon_c$ we have:

  $\foral R\ge 0, \exist L=L(\epsilon, R)
>0$ such that the condition  $\max\{d(\gamma_-, X),\; d(\gamma_+, X)\} < \epsilon$, for some $X \in \mathbb P$, implies that every point $ z \in \gamma$ satisfying
$d(z, \gamma_-),\; d(z, \gamma_+) >L$  is $(\epsilon_c, R)$-deep in $X$.
 \end{lem}

\begin{proof} The result is  proved in \cite[Lemma 2.8]{YANG7} for geodesics. We provide below a proof to precise  the choice of the constants.

By Lemma \ref{periphsystem}, let $\mu_c, D_c$ be the constants such that for any $X\in \mathbb P$, for any $c$-quasi-geodesic outside $N_{\mu_c}(X)$, the diameter of its projection to $X$  is upper bounded by $D_c$.

Set $\epsilon_c:=c(2\mu_c+D_c)+c$. If a $c$-quasi-geodesic has two endpoints in $N_{\mu_c}(X)$ for $X\in \mathbb P$, then it lies in $N_{\epsilon_c}(X)$. Indeed, if $x,y\in \gamma$ satisfy $$\max\{d(x, X), d(y, X)\}\leq\mu_c$$ and $]x, y[_\gamma\cap N_{\mu_c}(X)=\emptyset$, then   by  Lemma \ref{periphsystem},   $d(x,y)\leq 2\mu_c+D_c$. Since $\gamma$ is $c$-quasi-geodesic we have $\ell([x,y]_\gamma)\leq \epsilon_c$ and $\displaystyle [x,y]_\gamma\subset N_{\epsilon_c}(X)$.

Set $L= c(2\epsilon+D_c)+c+R$ for $\epsilon\ge \epsilon_c$. We first claim $\gamma\cap N_{\mu_c}(X)\ne \emptyset$. Otherwise, we obtain using projection the following $$2L\le \len(\gamma) \le  cd(\gamma_-, \gamma_+)+c\le c(2\epsilon+D_c)+c.$$ This gives a contradiction by the choice of $L$. Thus, there exist  the entry point $x$ and the exit point $y$ of $\gamma$ in
$N_{\mu_c}(X)$.

By the same argument one obtains  $$\max\{\len([\gamma_-, x]_\gamma), \len([y, \gamma_+]_\gamma)\}\le c(\epsilon + \mu_c+D_c)+c < L.$$ Since $\min\{d(z, \gamma_-),\; d(z, \gamma_+)\} >L$, we have $z\in [x,y]_{\gamma}$. Then we obtain $$\min\{d(x, z), d(z, y)\} \ge  L- (\epsilon + \mu_c+D_c)> R.$$

By  definition of $\epsilon_c$, we have $[x, z]_\gamma\subset N_{\epsilon_c}(X)$ and $[z, y]_\gamma\subset N_{\epsilon_c}(X)$. So $z$ is $(\epsilon_c,R)$-deep in $X$.
\end{proof}
\begin{rem}
By the proof, we actually have $\epsilon_c>\mu_c$, where $\mu_c$ is uniform for every $X\in \mathbb P$ by Lemma \ref{periphsystem}.
\end{rem}

In what follows, we take constants $\epsilon, R$ as in Convention \ref{epsilonR1}.
\begin{conv}[About $\epsilon_c, R_c$]\label{epsilonR1}
When talking about $(\epsilon_c, R_c, L)$-transitional $c$-quasi-geodesics, or $(\epsilon_c, R_c)$-transitional and $(\epsilon_c, R_c)$-deep points in a $c$-quasi-geodesic, we assume without explicitely specifying the quantifiers:
\begin{enumerate}
\item
$\epsilon_c=\epsilon(c)>\mu_c$ to satisfy Lemmas \ref{floydtrans} and \ref{uniformtrans}, where $\mu_c$ is given by Definition \ref{contractdefn}.
\item
$R_c>\mathcal R(\epsilon)$, where $\mathcal R$ is given by Lemma \ref{periphsystem}.
\end{enumerate}
\end{conv}

Besides the peripheral cosets (horospheres),  transitional quasi-geodesics provide another source of contracting subsets.
\begin{lem}[Transitional geodesic is contracting]\label{transcontract}
For any $L\ge 0$,   any $(\epsilon, R, L)$-transitional geodesic $\gamma$ is $1$-contracting.
\end{lem}
\begin{rem}
The same  argument also applies (with natural changes for the constants) to show that if $\gamma$   is $c$-quasi-geodesic then it is contracting for any $c\ge 1$.  For our purposes, we only need to consider the case when $c=1$.
\end{rem}
\begin{proof}
Let $\kappa=\kappa(\epsilon, R)$ given by Lemma \ref{floydtrans} and $\phi$ given by Lemma \ref{karlssonlem}. By Lemma \ref{karlssonlem}, there exists $D_0=\phi(\kappa/4)>0$ such that for any $v\in G$, a geodesic segment outside the ball $B(v, D_0)$ has $\fls_v$-Floyd length less than $\kappa/4$.

Let $D=2(L+2D_0+1)$ and $\mu=\phi(\kappa/2)$. Let $\beta$ be a geodesic such that $\beta \cap N_\mu(\gamma)=\emptyset$. Let $x, y \in \proj_\gamma(\beta)$ such that $d(x, y)=\diam{\proj_\gamma(\beta)}$. We are going to prove that $d(x, y) \le D$. Suppose by contradiction that $d(x, y) > D$.

Assume that $x, y$ are projection points of $\tilde x, \tilde y\in \beta$ respectively. Observe that
\begin{equation}\label{projtriangle}
2d(z, [x, \tilde x])\ge d(z, x), \;
2d(z, [y, \tilde y])\ge d(z, y),
\end{equation}
for any $z\in [x, y]_\gamma$. We only prove the first inequality; the second one is completely analogous. Let $m\in [x, \tilde x]$ such that $d(z,m)=d(z, [x, \tilde x])$. Note that $d(m, z)+d(m, \tilde x)\ge d(x, \tilde x)$  by the shortest point property. Since $d(x, \tilde x)= d(x, m)+d(m, \tilde x)$ we obtain $d(m, z)\ge d(x, m)$. Then $d(z, x) \le d(z, m)+d(m, x) \le 2d(z, m)$ implying (\ref{projtriangle}).

\medskip

 \centerline{\includegraphics[width=8.5cm, height=4.5cm]{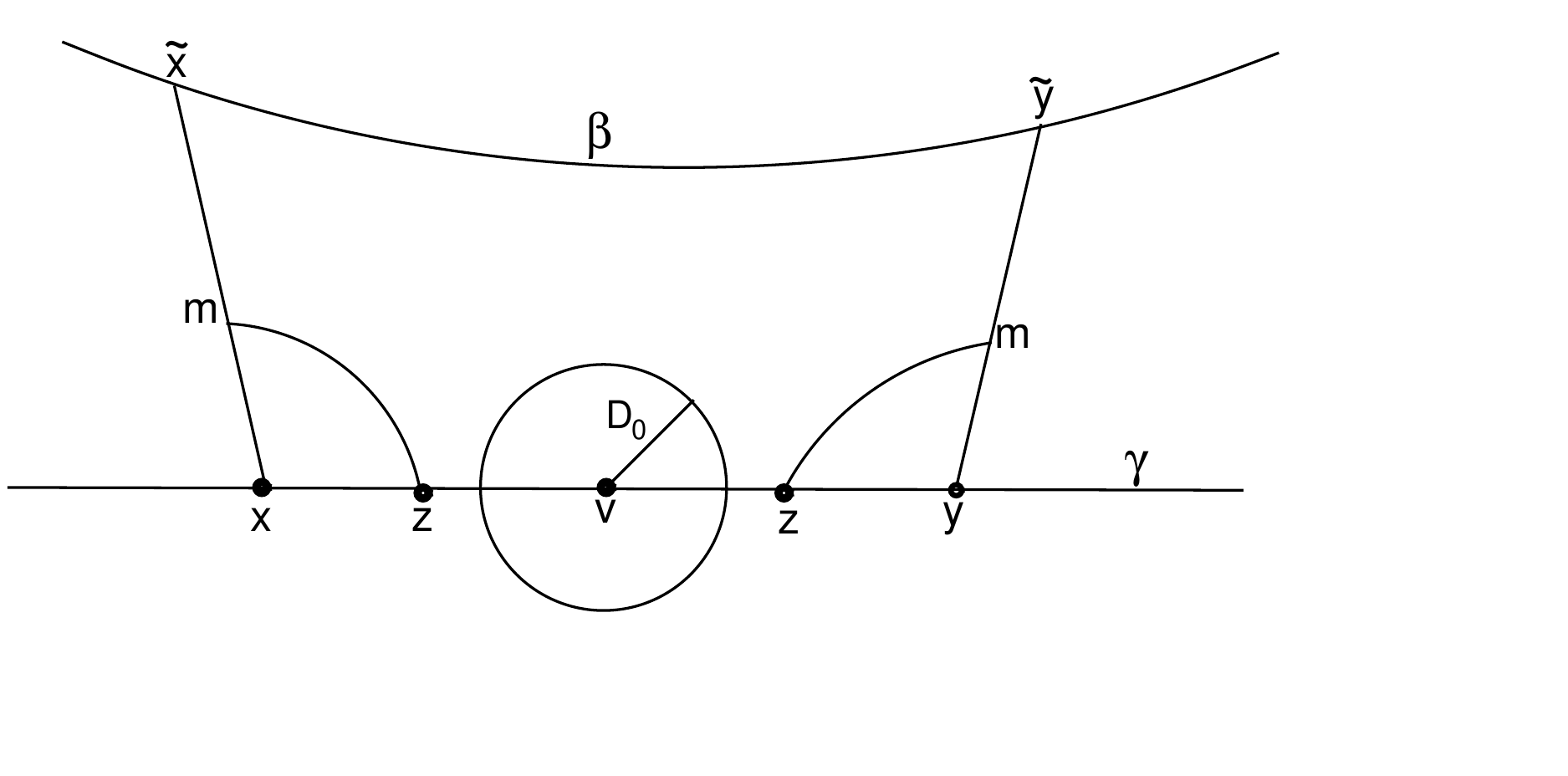}}

Since $d(x, y) > D$, there exists $z\in [x,y]_\gamma$ such that
$$\min \{d(z,x), d(z,y)) > D/2=L+2D_0+1.$$ Since $\gamma$ is $(\epsilon, R, L)$-transitional, one of the intervals $[x,z]_\gamma$ or $[z,y]_\gamma$ contains an $(\epsilon,R)$-transitional point $v$  such that $\min \{d(x,v), d(y,v)\} >2D_0.$ Hence by (\ref{projtriangle}), $\min \{d(v, [x, \tilde x]), d(v, [y,\tilde y]))\} > D_0.$ By the choice of $D_0=\phi(\kappa/4)$, we have $$\max\{\rho_v(x, \tilde x), \rho_v(y,\tilde y)\} <  \kappa/4.$$ From the other hand, $v$ is $(\epsilon, R)$-transitional, so $\rho_v(x,y)\ge\kappa$ by Lemma \ref{floydtrans}. Hence,  $\rho_v(\tilde x, \tilde y) > \kappa/2$ and thus $d(v, \beta) \leq\mu$ which is impossible.
\end{proof}

For a c-quasigeodesic we denote by $\epsilon_c=\epsilon(c), R_c=\mathcal R(\epsilon_c)$ any numbers  satisfying  Convention \ref{epsilonR1} (in particular $\epsilon_1$ and $R_1$ correspond to geodesics).
In the following Proposition we will establish a "thinness" of a triangle whose two sides are  transitional geodesics.

\begin{prop}[Transitional triangle is thin]\label{tightbetweenconical}
For any $L, c>0$ and the there exist constants $D=D(c), M=M(L, c), L'=L'(L, c)>0$ with the following properties.

Let $\alpha_1, \alpha_2$ be $(\epsilon_1, R_1,
L)$-transitional geodesic rays issuing at $o$ and ending at $\xi \ne \eta\in \pG$ respectively.  Then for any $c$-quasi-geodesic $\gamma$ with $\gamma_-\in \alpha_1, \gamma_+\in \alpha_2$, the following holds.
\begin{enumerate}
\item
$\gamma$ is $(\epsilon_c, R_c, L')$-transitional.
\item
If the length of $\gamma$ is sufficiently large then there exists an $(\epsilon_c, R_c)$-transitional point $z \in \gamma$ such that $d(z, \alpha_1\cup\alpha_2) \le D$ and $d(z, \alpha_i)\le M$ for $i=1, 2$.
\item
Let $d(o, [\xi, \eta])$ denote the distance from $o$ to a geodesic between $\xi, \eta$.  If $\min\{d(\gamma_-, o), d(\gamma_+, o)\} \gg 0$. Then
$|d(o, [\xi, \eta]) -d(o, \gamma)| \le M$.
\end{enumerate}
\end{prop}

\begin{rem}
 Note that in (2) $D$ is a uniform constant not depending on $L$, this will play a crucial role in establishing Lemma \ref{FloydMetric} below.
\end{rem}

\begin{proof}
Let $\kappa=\kappa(\epsilon_c, R_c)$ given by Lemma \ref{floydtrans} and $D=\phi(\kappa/2)$, where $\phi$ is given by Lemma \ref{karlssonlem}.  The constant $L'$ will be computed below.

\textbf{(1)}  Given a point $x$ in $\gamma$,  assume that $x$ is $(\epsilon, R_c)$-deep in some $X \in \mathbb P$. Let $x_-, x_+$ be the entry and exit points of $\gamma$ in $N_{\epsilon_c}(X)$ respectively.

Observe first that $x_-, x_+$ are $(\epsilon_c, R_c)$-transitional in $\gamma$. Indeed, if not, there exists $Y\in \mathbb P$ such that $x_-$ is $(\epsilon_c, R_c)$-deep in $Y.$ Then $Y\ne X$ by the choice of $x_-$ as the entry point of $\gamma$ in $N_{\epsilon_c}(X)$. Since $d(x, x_-)\ge R_c$, we have $\diam{N_\epsilon(X) \cap N_\epsilon(Y)} \ge R_c > \mathcal R(\epsilon_c)$ by Convention \ref{epsilonR1}. This contradicts to Lemma \ref{periphsystem}.

 To find a constant $L'$  we  will check when   the opposite inequality: \begin{equation}\label{opineq}\min\{\len([x, x_-]_\gamma), \len([x, x_+]_\gamma)\}> L'\end{equation}
 is not valid.
  We have  $\len([x_-, x_+]_\gamma)\ge 2L'$.
Since $x_-, x_+$ are $(\epsilon_c, R_c)$-transitional, by Lemma \ref{floydtrans}, $$\min\{\fds_{x_-}(\gamma_-, \gamma_+), \fds_{x_+}(\gamma_-, \gamma_+) \}>\kappa.$$ By the triangle inequality $$\max\{\fds_{x_-}(\gamma_-,o), \fds_{x_-}(o,\gamma_+)\}\geq {\kappa\over 2},$$ and the same for $\fds_{x_+}.$ Then $\max\{d(x_-, \alpha_1\cup\alpha_2), \; d(x_+, \alpha_1\cup\alpha_2)\} \le D=\phi(\kappa/2)$. For concreteness consider the case that
\begin{equation}\label{EQ2}
d(x_-, \alpha_1), \;d(x_+, \alpha_2) \le D;
\end{equation}
the other cases are similar and even easier.

Project $x_-, x_+$ to $x'_-, x'_+ \in X$ so that $d(x_-, x'_-), \; d(x_+, x'_+) \le \epsilon_c.$  So $$d(x'_-, \alpha_1),\; d(x'_+, \alpha_2) \le \epsilon_c+D$$ and $N_{D+\epsilon_c}(X)\cap \alpha_i \ne \emptyset\ (i=1,2)$.

Let $w \in X$ be a projection point of $o$ to $X$. We claim that
\begin{equation}\label{EQ2'}
d(w, \alpha_i) \le D_2:=\max\{D+\epsilon_c+D_1, \mu_1+D_1\}\ (i=1,2),
\end{equation}
  where $\mu_1, D_1>0$ are given for $1$-contracting $X\in \mathbb P$ such that  (\ref{contractingproperty}) holds.

  Indeed if, first, $o\in N_{\mu_1}(X)$ then there is nothing to prove. If not, there are two more cases:  if $\alpha_i\cap N_{\mu_1}(X)=\emptyset$, then by the contracting property we have $d(w, \alpha_i) \le D+\epsilon_c+D_1$; otherwise the projection on $X$ of the maximal connected subcurve of $\gamma$, situated  outside of $N_{\mu_1}(X)$ and containing $o$,  gives $d(w, \alpha_i) \le \mu_1+D_1$. So (\ref{EQ2'}) follows.

Let $L_0=L(\epsilon_c+D+D_2, R_1+L)$ given by Lemma \ref{uniformtrans}. Set
\begin{equation}\label{needv}L'=2c(D_2+D+L_0+2\epsilon_c)+c^2.\end{equation}

Since $\gamma$ is a $c$-quasi-geodesic, we have
$$
\begin{array}{lll}\label{EQ1}
d(x'_-, x'_+)\ge \len([x_-,x_+]_\gamma)/c-c- 2\epsilon_c&\ge 2L'/c -c- 2\epsilon_c\\
&\ge 4(\epsilon_c+D+D_2+L_0).
\end{array}
$$

 Since $\max\{d(x'_-, \alpha_1),\; d(x'_+, \alpha_2)\} \le \epsilon_c+D$, we obtain from  (\ref{EQ2'}) $$\diam{\alpha_1\cap N_{\epsilon_c+D+D_2}(X)} \ge d(x'_-, w),\; \diam{\alpha_2\cap N_{\epsilon_c+D+D_2}(X)} \ge d(x'_+, w).$$
We have $d(x'_-, w)+d(x'_+, w)\ge d(x'_-, x'_+)\ge 4(\epsilon_c+D+D_2+L_0)$.  Thus
$$
\max_{i=1,2}\diam{\alpha_i\cap N_{\epsilon_c+D+D_2}(X)} \ge 2(\epsilon_c+D+D_2+L_0).
$$

Hence, $\alpha_i$ contains a subcurve of length at least $2L_0$  such that its endpoints lie in $N_{\epsilon_c+D+D_2}(X)$. By the choice of $L_0$ and Lemma \ref{uniformtrans}, $\alpha_i$ contains an $(\epsilon_1, R_1+L)$-deep point  in $X$. This gives a contradiction, as $\alpha_i$ is $(\epsilon_1, R_1, L)$-transitional. So for the value of $L'$ chosen in (\ref{needv}) the  inequality  (\ref{opineq}) is not valid. The statement (1) is proved.

\textbf{(2)}  By the statement (1) $\gamma$ is $(\epsilon_c, R_c, L')$-transitional.  Lemma \ref{transcontract} implies that $\gamma$ is   contracting. By the projection argument (used to prove  (\ref{EQ2'})) we have a constant $D_3=D_3(\epsilon_c, R_c, L')>0$ such that for any projection point $v$ of $o$ to $\gamma$ we have $d(v, \alpha_i) \le D_3$ for $i=1, 2$.

\begin{rem} We need a new constant $D_3$ (and not $D_2$ used above) since we project now on $\gamma$ and not on a horosphere.
\end{rem}

Recall that $D=\phi(\kappa/2)$. By Lemma \ref{karlssonlem}, for any $z\in G$, a geodesic segment outside $B(z, D)$ has $\fls_z$-Floyd length less than $\kappa/2$.

The curve $\gamma$ is quasigeodesic and its length is sufficently large. So by continuity of the distance function $d(v,x)\ (x\in \gamma)$ we find a point $z'$ such that  $D+D_3+L'\leq d(v, z')\leq D+D_3+L'+1.$  Since $\gamma$ is $(\epsilon_c, R_c, L')$-transitional by Definition \ref{trans}  there exists an $(\epsilon_c, R_c)$-transitional point $z\in\gamma$ for which $d(z',z)\leq l([z,z']_\gamma)\leq  L'$. We obtain
\begin{equation}\label{EQ4}
D+D_3 \le d(v, z) \le  D+D_3+2L'+1.
\end{equation}

 Then $$d(z, \alpha_i) \le d(z, v)+d(v, \alpha_i)\le M,$$

\noindent where $M=2L'+2D_3+D+1.$

To prove the first claim of {\bf (2)} assume for definiteness that $z \in [v, \gamma_+]$.  Lemma \ref{floydtrans} yields $\fds_z(v, \gamma_+) \ge \kappa$.

Let $z_2 \in \alpha_2$ such that $d(v, z_2) \le D_3$. We have $d(z,[v,z_2])\geq d(z,v)-d(z_2,v)\geq D+D_3-D_3=D.$   By Lemma \ref{karlssonlem} $\fds_z(v, z_2) < \kappa/2$ and so $\fds_z(z_2, \gamma_+)\ge \fds_z(v, \gamma_+)-\fds_z(v,z_2)\geq\kappa/2$. Lemma \ref{karlssonlem} gives $$d(z, \alpha_2) \le D.$$   The statement (2) is proved.

\textbf{(3)}  Since $\xi, \eta$ are distinct, by  Lemma \ref{karlssonlem}, there exists $n_0=n_0(\xi, \eta), r=r(\xi, \eta)>0$ such that if  $$\min\{d(\gamma_-, o), d(\gamma_+, o)\}>n_0,$$ then $d(o, \gamma)\le r$. In the proof of the statement (2), we projected $o$ to a point $v$ in $\gamma$, and found an $(\epsilon, R)$-transitional point $z \in \gamma$ such that $d(v, z) \le M$.

Since $d(o, z)\le M+r$ (and these constants do not depend on $\gamma$)  up to increasing $n_0$, by Lemma \ref{karlssonlem} we have  $$\max\{\bar \fds_z(\xi, \gamma_-),\bar \fds_z(\eta, \gamma_+) \}\le \kappa/4.$$ The point $z$ is $(\epsilon, R)$-transitional, thus $\bar\fds_z(\gamma_-, \gamma_+)\ge \kappa$, and so $\bar\fds_z(\xi, \eta) \ge \kappa/2$. Consequently  $d(z,  [\xi, \eta]) \le D$ which yields: $$
\begin{array}{lll}
d(o,  [\xi, \eta])& \le d(o, z) +  d(z,  [\xi, \eta])\le d(o, \gamma) + M+ d(z,  [\xi, \eta]) \\
&\le d(o, \gamma) + M+ D.
\end{array}$$
By symmetry, we obtain  $d(o, \gamma) \le d(o, [\xi, \eta]) + M+D$.

Since $D$ is a uniform constant not depending on $L$ we put $M:=M+D$. Then the statements of (2) and (3) are both valid for the same constant $M.$
The Proposition is proved.
\end{proof}

 The  claim (3) of the Proposition and Lemma \ref{charuc} imply:
\begin{cor} Suppose $\GP$ is a relatively hyperbolic pair.
Then for any $L>0$, there exists $M=M(L)$ such that for any $\xi, \eta \in \uGLo$ or $\xi, \eta \in \uGfLo$, the distance $d(o, [\xi, \eta])$ is comparable with the distance $d(o,\gamma)$ where $\gamma$ is a $c$-quasi-geodesic with the endpoints on the corresponding geodesic rays converging to $\xi$ and $\eta.$.
\end{cor}

\section{Patterson-Sullivan measures on ends of a geodesic tree}\label{Section4}

In this section, we shall construct an iterated transitional tree having several nice properties which will allow us to carry out the Patterson's construction on this tree. The space of ends  of the tree equipped with the Patterson-Sullivan measure will give rise to an Ahlfors regular subset of the boundary.

\subsection{Iterated Transitional Trees}

Let $\GP$ be a relatively hyperbolic pair and $\Gx$ the Cayley graph of $G$ with respect to $S$. The existence of large transitional trees is established in \cite[Theorem 5.9]{YANG7}. The main difference of the construction below is that these trees will be equipped with certain periodicity. By this reason we call them \textit{iterated transitional trees}.
We start by recalling several results from \cite{YANG7}.
\begin{defn}[Partial Cone]\label{pconedef}
For $\epsilon, R
\ge 0$, the \textit{partial cone $\Omega_{\epsilon, R}(g)$} at $g \in G$ is the set of elements $h \in G$ such that there exists a geodesic $\gamma=[1, h]$ containing $g$ and one of the following holds.
\begin{enumerate}
\item $d(1, h) \le d(1, g) + 2R,$
\item
$\gamma$ contains an $(\epsilon, R)$-transitional point $v$ such that $d(v, g) \le 2R$.
\end{enumerate}
\end{defn}

For $\Delta \ge 0, n \ge 0$, define
$$
A(g, n, \Delta) = \{h \in G:  n-\Delta \le d(1, h) - d(1, g) <
n+\Delta\},
$$
for any $g \in G$. For simplicity we write $A(n, \Delta):=A(1, n,
\Delta)$. For $r, \epsilon, R, \Delta > 0$, define
$$
\Omega_{\epsilon, R}(g, n, \Delta) = \Omega_{\epsilon, R}(g)
\cap A(g, n, \Delta),
$$
for any $g \in G, n \ge 0$.

For fixed $\epsilon, R>0$,  two partial cones
$\Omega_{\epsilon, R}(g), \Omega_{\epsilon, R}(g')$ are of
 \textit{same type} if   $$g'g^{-1}\cdot
\Omega_{\epsilon, R}(g)= \Omega_{\epsilon, R}(g').$$ By abuse of language, we  say that $g, g'$ have the same partial cone types.

The following result generalizes the result of Cannon \cite{Cannon} for hyperbolic groups.

\begin{lem}[Finiteness of partial cone types]\label{FinPCones}\cite[Lemma B.1]{YANG7}
There exist $\epsilon, R_0>0$ such that for any $R > R_0$, there are
at most $M=M(\epsilon, R)$ types among all $(\epsilon, R)$-partial
cones $\{\Omega_{\epsilon, R}(g):  g \in G\}$.
\end{lem}

The following is a key technical result in \cite[Lemma 5.8]{YANG7}.

\begin{lem}\label{TreeSetG}
There exist $\epsilon, R, \Delta, \theta, L_0>0$ with the following
property.

For any $L
>L_0$ there exists a subset $\hat G$ of $G$ such that
\begin{equation}\label{LBound}
\sharp(\Omega_{\epsilon, R}(g, L,\Delta) \cap \hat G) > \theta \cdot\exp(L\cdot \g G),\; 1\in \hat G
\end{equation}
for any $g \in \hat G$.
\end{lem}

\begin{conv}[$\epsilon, R, \Delta$]\label{epsilonRDelta}
Until the end of Section \ref{Section4}, the constants $\epsilon, R, \Delta>0$ are given by Lemmas \ref{FinPCones} and \ref{TreeSetG}, and satisfy Convention \ref{epsilonR1}.
\end{conv}

The following terminology comes from the paper \cite{BishopJ} which was certainly very motivating for us.

\begin{defn}[Iterated Tree Set]\label{treeset}
For given $L>0$, an \textit{$L$-iterated tree set} $T$
in $G$ is a union
of a sequence of sets $T_i$ ($i\ge 0$) in $G$ defined
inductively as follows.

Let $T_0=\{1\}$. Assume that $T_i$ is defined for $i\ge 0$. The \textit{children} $T(x)$ of  $x\in T_i$ is a subset in $\Omega_{\epsilon,
R}(x, L, \Delta)$. Then $T_{i+1}$ is the union of children
of all $x\in T_i$.
\end{defn}

Recall that a subset $Z$ of a metric space space $(X,d)$ is called {\it $C$-separated} if the distance$d(z_1,z_2) \geq C$ for every pair of distinct points $\{z_1, z_2\}\subset Z.$ The following fact is elementary.
\begin{lem}\label{sepnet}
Let $(X, d)$ be a proper metric space on which a group $G \subset \isom(X)$ acts properly on $X$. For any orbit $Go$ ($o\in X$) and $C>0$ there exists a constant $\theta=\theta(Go, C)
>0$ with the following property.

For  any finite set $Y$ in $Go$, there exists a $C$-separated
subset $Z \subset Y$ such that $\sharp Z \ge \theta \cdot \sharp Y$.
\end{lem}
\begin{proof}
Let $Z$ be a maximal $C$-separated set in $Y$. We have $Y\subset N_C(Z)$.
Since the action of $G$ on $(X, d)$ is proper, any ball of radius $C$ contains at most $N$ points in $Go$. The result follows for $\theta:=1/N$.
\end{proof}

An \textit{$(\epsilon, R, L)$-transitional geodesic tree $\mathcal T$} rooted at $o$ in $\Gx$ is a tree subgraph with a distinguished vertex $x$ such that  every branch in $\mathcal T$ originating at $x$ is a $(\epsilon, R, L)$-transitional geodesic in $\Gx$.


In order to obtain a useful theory of Patterson-Sullivan measures,  certain symmetry on the iterated tree set is required. This is the content of the following.
\begin{lem}[Existence of iterated transitional trees]\label{ITT}
There exist constants $L_0, C_0, t_0, n_0>0$ such that for $L>L_0, C>C_0$ there are  $\theta=\theta(C)$ and $L'=L'(L)$ and an iterated tree set $T$ parameterized by
$(\epsilon, R, L)$ with the following properties:
\begin{enumerate}
\item
$x^{-1}T(x)=y^{-1}T(y)$ for any $x \in T_t, y\in T_{t+n_0}$ and $t\ge t_0$,
\item
$x^{-1}T(x)=y^{-1}T(y)$ for any $x, y \in T_t$ and $t\ge t_0$.
\item
$\sharp T(x) \ge \theta \cdot \exp(\g G L)$ for any $x\in T$.
\item
$T(x)$ is $C$-separated for any $x\in T$.
\item
there exists an $(\epsilon, R, L')$-transitional geodesic tree
$\mathcal T$ rooted at $1$ in $\Gx$ such that the vertex set $\mathcal T^0$ contains $T$, and lies in $N_{L'}(T)$.
\end{enumerate}
\end{lem}
\begin{proof}
Set $L_0=\Delta$, and all other constants will be defined in the proof. We divide the proof into 3 steps for the reader convenience.

\textbf{Step 1:}
At this step we construct the iterated tree set $T$ with properties (1-3). The construction proceeds by an induction argument. Set $T_0=\{1\}$ to start.

Let $M$ be the number of $(\epsilon, R)$-partial cone types in $G$ given by Lemma \ref{FinPCones}, and $\hat G$ the set   by Lemma \ref{TreeSetG}.  Then there exists $T_1 \subset \Omega_{\epsilon, R}(1, L, \Delta) \cap \hat G$ such that every element in $T_1$ has the same partial cone type and the inequality (\ref{LBound}) holds for $g=1$ where the constant $\kappa$ is divided by $M$. By Lemma \ref{sepnet}, we can also require that $T_1$ is $C$-separated, where $\kappa$ is further decreased and depends on $C$ (given in Step 3 below).

Fix some $x_1\in T_1$.  Up to dividing $\theta$ by $M$ again, we choose $Y$ to be a subset of $\Omega_{\epsilon, R}(x_1, L, \Delta)\cap \hat G$ such that the inequality (\ref{LBound}) holds for $Y$ and every element in $Y$ has the same partial cone type. By the same reason, we can choose  $Y$ to be $C$-separated. Since all $x\in T_1$ are of same type as $x_1$,  we could define $$T(x):=xx_1^{-1} Y \subset \Omega_{\epsilon, R}(x, L, \Delta).$$ Then all elements in the union $T_2 :=\cup_{x\in T_1}T(x)$ have the same partial cone types. We note that $Y$ is chosen to be contained in $\hat G$, but $T_2$ may not be in $\hat G$.

We repeat  the same argument to construct $T_i$ for $i\ge 3$, with a sequence of divisions of $\kappa$. By construction, all elements in the constructed $T_i$ are of the same partial cone type. Since there are finitely many partial cone types, we obtain that there are $1 \le t_0, n_0 \le M$ so that $x^{-1}T(x)=y^{-1}T(y)$ for any $x \in T_t, y\in T_{t+n_0}$ and $t\ge t_0$. This also implies that the division of $\theta$ stops after at most $n_0$ times, and thus $\theta$ in the inequality (\ref{LBound}) can be chosen uniformly for all $T(x)$ where $x\in T_i$ $(i \ge 1)$. The set $T$ satisfies the properties ($1-3$).

\textbf{Step 2:} Using the iterated tree set $T$, we will now construct a geodesic  graph $\mathcal T$.

Without loss of generality assume that $\kappa < 1$.  The root of
$\mathcal T$ is  $\mathcal T_0=\{1\}$.  Assume that $\mathcal T_i$ is defined for $i\ge 0$ and for each terminal vertex $x \in \mathcal T_i$, denote by $\gamma_x$ the geodesic $[1, x]$ in $\mathcal T_i$. We choose a geodesic $[x, y]$ for each $y\in T(x)$. Since $T_i(x)$ is a subset in $\Omega_{\epsilon, R}(x, L,\Delta) \cap \hat G$, we set
\begin{equation}\label{GeodTree}
\mathcal T_{i+1} = \cup_{x \in \mathcal T_i} (\cup_{y \in
T(x)}\gamma_x \cdot [x, y]),
\end{equation}
where $\gamma_x \cdot [x, y]$ is a geodesic in $\Gx$. Inductively, we get the limit $\mathcal T =\lim_{i\to
\infty} \mathcal T_i$. By construction, each geodesic ray originating at $1$ is $(\epsilon, R, L')$-transitional for $L':=L+2R+\Delta$. By construction we have  $T\subset \mathcal T^0\subset N_{L'}(T)$.

\textbf{Step 3:} We now prove that $\mathcal T$ is a geodesic tree rooted at $1$ in $\Gx$. Indeed, if not, there exist two distinct geodesics $\alpha_1, \alpha_2$ in $\mathcal T$ with the same endpoints $x, w \in T$ such that the length of $\alpha_1, \alpha_2$ is minimal among all such choices. Assume that $x$ is closer to $1$ than $w$. Consider two points $y_i \in \alpha_i \cap T(x)$ for $i=1, 2$. By the choice of $\alpha_1, \alpha_2$, we have $y_1\ne y_2$.  Then by construction, $d(y_i, w)\ge L-\Delta$ for $i=1, 2$. Moreover, there exists an $(\epsilon, R)$-transitional point $z_1\in \alpha_1$ such that $d(y_1, z_1) \le 2R$.

Let $D_0=\phi(\kappa)$, where $\phi$ is given by Lemma \ref{karlssonlem} and $\kappa=\kappa(\epsilon, R)$ by Lemma \ref{floydtrans}.  There exists $z_2 \in \alpha_2$ such that $d(z_1, z_2) \le D_0$ and then $d(y_1, z_2) \le 2R+D_0$. We can choose $ \tilde y_2\in \alpha_2$ such that $d(x, \tilde y_2) = d(x, y_1)$. Hence $d(z_2, \tilde y_2)=|d(x, z_2)-d(x, \tilde y_2)|=|d(x, z_2)-d(x, y_1)|\le 2R+D_0$. It follows that $d(y_1, \tilde y_2) \le 2(2R+D_0)$.

Since $y_1, y_2$ lie in the annulus $A(x, L, \Delta)$, we get $|d(x, y_1)-d(x, y_2)| \le 2\Delta$, and then $d(y_2, \tilde y_2) = |d(x, \tilde y_2)-d(x, y_2)| \le 2\Delta$. It follows that $d(y_1, y_2) \le 2(2R+D_0+\Delta)$. Choosing now the constant $C$ to be greater than
\begin{equation}\label{consep}
C_0:=2(2R+\phi(\kappa/2)+\Delta)\end{equation}
we obtain  that $T(x)$ is $C$-separated in $\Omega_{\epsilon,
R}(x, L, \Delta)$, and $d(y_1, y_2) \ge C_0 > 2(2R+D_0+\Delta)$ which  is  a contradiction. Thus, $\mathcal T$ is a rooted geodesic  tree.
\end{proof}
\begin{rems} \begin{enumerate}\item By Lemma \ref{charuc} the boundary  of the tree $\mathcal T$ (in $\pG$ or in $\pGf$) constructed above consists of uniformly conical points.
\item The constant $C_0$ in (\ref{consep}) is bigger than we really need in the above proof (it is enough to replace
$\varphi(\kappa/2)$ by the smaller term $D_0=\varphi(\kappa))$) but we do need such a constant in the next lemma.
\end{enumerate}
\end{rems}

In the next two lemmas, we shall derive more properties of the  sets $T$ and $\mathcal T$ constructed in Lemma \ref{ITT}. To this end, we recall the notion of Poincar\'e series.

For a
subset $X \subset G$ and a point $o\in G$, set
$$\PS{X}{o} = \sum\limits_{g \in X} \exp(-sd(o, g)), \; s \ge 0.$$
Define the \textit{critical exponent} of
$\PS{X}{o}$ to be
\begin{equation}\label{critexp}\g X=\limsup\limits_{n \to \infty} \frac{\log \sharp (B(o, n)\cap X)}{n},\end{equation}
where $S$ is a fixed finite symmetric generating set of $G$, and $B(o,n)$ is the ball in the word metric of radius $n$ centered at $o$.

It is elementary fact that $\PS{X}{o}$ converges for $s> \g X$, and diverges for $s<\g X$.

Recall that the bilipschitz equivalence $\asymp_{\rm const}$ between two functions means that they are comparable up to a constant (see Section \ref{Section2}). We have the following.

\begin{lem}\label{PoincareSeries}
Under the same assumptions as in Lemma \ref{ITT},  we have
$$\PS{T}{x} \asymp_{L} \PS{T}{y}$$ for any $x, y\in T$ and $s\ge 0$, whenever one of the series converges.
\end{lem}
\begin{proof}
 Let $\Omega(x)$ be a {\it cone} at $x\in T$, which is the union of $y\in T$ such that the unique geodesic $[1, y]$ in the geodesic tree $\mathcal T$ contains $x$.

\begin{claim}
The Poincar\'e series of $T$ is bilipschitz equivalent to that of any cone at a vertex in $T$:
\begin{equation}\label{BiLipCone}
\PS{T}{1} \asymp_{L} \PS{\Omega(x)}{x},
\end{equation}
for any $x \in T$.
\end{claim}
\begin{proof}[Proof of the Claim]
It follows from Lemma \ref{ITT}.(1) that after a finite time $t_0$, the  set  $T$ is periodic with a fixed period $n_0$.
So it is enough to show (\ref{BiLipCone}) for $x\in T$ such that $t_0 \le d(1,x)\le n_0+t_0.$ By Lemma \ref{ITT}.2, the cones  based at  points $y\in T_t$  have the same type where $t=d(1,x)$. Thus the number $a_n$ of points in $T$ situated at the distance $n$ from $1$ is at most $C\cdot b_{n,x}$. Here $C$ is the number of elements in the ball $B(1,t),$ and $b(n,x)$  is the number of elements of $\Omega(x)$ at the distance $n$ from $x$. The same argument works in the opposite sense. The Claim follows.\end{proof}

To complete the proof of the lemma, by (\ref{BiLipCone}), it suffices to establish the following \begin{equation}\label{BiLipCone2}
\PS{T}{1} \asymp_{L} \PS{T\setminus \Omega(x)}{x},
\end{equation}  as \ref{BiLipCone} and \ref{BiLipCone2} would imply $\PS{T}{1}\asymp_L\PS{T}{x},\foral x\in T.$

 For  $y\in T\setminus \Omega(x)$, let $o$ be the farrest point to $1$ such that $o\in T$ and $[1,o]\subset [1, x]\cap[1, y]$, where the geodesics $[1, x]$ and $[1, y]$ are in the geodesic tree $\mathcal T$. The point $o$ will be referred to as the \textit{branch point} of  $[1, x]$ and $[1, y]$.

By Lemma \ref{ITT}, $[1, x], [1, y]$ are $(\epsilon, R, L')$-transitional geodesics, where $L'=L'(L)$. By Proposition \ref{tightbetweenconical}, $[x, y]$ is transitional and so is contracting by Lemma \ref{transcontract}.

\begin{claim}
There exists a uniform constant $D=D(L)>0$ such that $d(o, [x, y]) \le D$.
\end{claim}
\begin{proof}[Proof of the Claim]

Let $z \in [x, y]$ be the projection   of $o$ to a geodesic $[x, y]$ in the Cayley graph $\Gx$. By the contracting property of $[x, y]$ it follows from the inequality \ref{EQ2'} that there exists $D_1=D_1(\epsilon, R, L')$ such that $$\max\{d(z, [o, x]), d(z, [o, y])\} \le D_1.$$
So, let $x_1 \in [o, x], y_1\in [o, y]$ such that $d(z, x_1)\le D_1$ and $d(z, y_1)\le D_1$.

Set  $d(o,z)=d$, then
 \begin{equation}\label{opposite}
 \min\{d(o, x_1), d(o, y_1)\} \geq d-D_1
 \end{equation}


Let $w \in [o, x_1]\cap T(o)$ where $T(o)\subset \Omega_{\epsilon, R}(o, L, \Delta)$. Then $d(o, w)<L+\Delta$. Furthermore since $x\in T$  there exists an $(\epsilon, R)$-transitional point $x_2\in  [o, x] $ such that $d(w, x_2) \le 2R$, and so $d(o, x_2) \leq L+\Delta+2R.$ Using (\ref{opposite})
we deduce \begin{equation}\label{cruconst} d(x_2, [x_1,  y_1])\geq d(x_1,o)-d(x_2,o)-2D_1\geq K,\end{equation}   \noindent where $K=d-3D_1-L-\Delta-2R.$

We affirm  that \begin{equation}\label{hypothet}
K\leq \phi(\kappa/2),
\end{equation}
where $\kappa$ and $\phi$ are universal constants given by Lemmas \ref{floydtrans} and  \ref{karlssonlem}.
respectively.
  Indeed, suppose (\ref{hypothet}) is not true, then $d(x_2, [x_1, y_1]) \ge \phi(\kappa/2)$. By Lemma \ref{karlssonlem} we have $\fds_{x_2}(x_1, y_1)\le \kappa/2$. Since $x_2$ is transitional,    Lemma \ref{floydtrans}  yields $\fds_{x_2}(o, x_1) \ge \kappa.$  It follows $\fds_{x_2}(o, y_1) \ge \kappa/2$, and thus  $\exist \tilde x_2\in [o,y]\ :\  d(x_2, [o,y])=d(x_2, \tilde x_2) \le \phi(\kappa/2)$.

\medskip

 \centerline{\includegraphics[width=6.5cm, height=6.5cm]{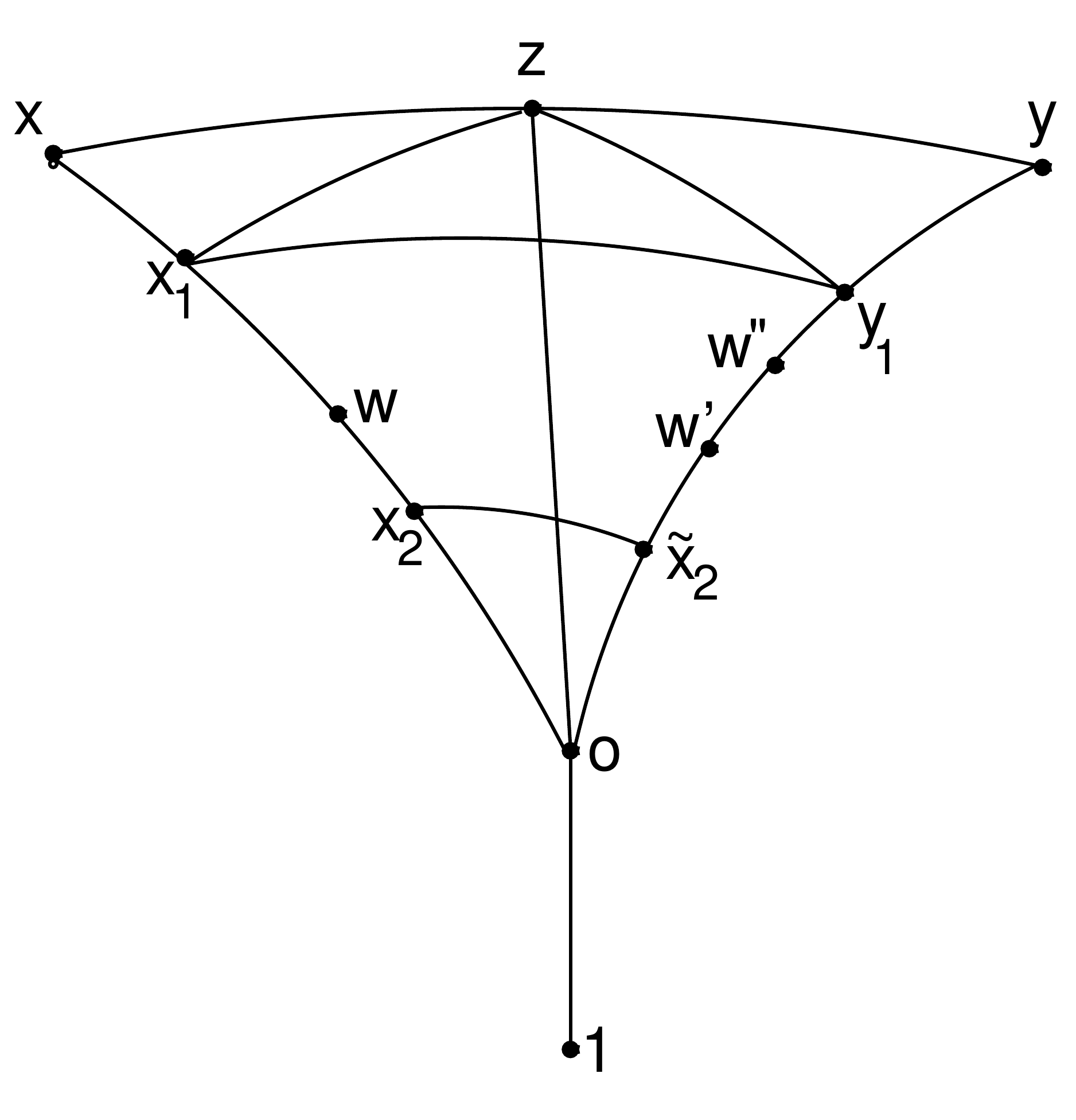}}

 Following the argument of Step (3) of Lemma \ref{ITT} we choose a vertex $w'\in [o,y]$ such that $d(o,w')=d(o,w).$   Then we have $d(\tilde x_2, w')=\vert d(o,\tilde x_2)-d(o,w)\vert\leq d(\tilde x_2,w)\leq 2R+\phi(\kappa/2).$ Then $d(w',w)\leq d(w',\tilde x_2)+d(\tilde x_2,w)\leq 2(2R+\phi(\kappa/2)).$ Let $w''\in T(o)\cap [o,y]$. Since  $w\in\Omega_{\epsilon, R}(o, L, \Delta)$  we have $d(w'',w')\leq \vert d(o,w') - d(o, w'')\vert\leq  2\Delta$. Indeed $\vert (d(o,w')=d(o,w)) -L\vert\leq \Delta$ and also $\vert d(o,w'')-L\vert\leq\Delta.$ Therefore for the vertices  $w'', w\in T(o)\subset T$ we have $d(w,w'')\leq  2(\phi(\kappa/2)+R+\Delta).$ This is impossible by (\ref{consep}). The obtained contradiction implies that  $K\leq \phi(\kappa/2)$ and by definition of $K$ (see (\ref{cruconst})), we have

  $$d(o,z)=d\leq D=D_1+\phi(\kappa/2)+L+\Delta+2R.$$
The claim is proved.
\end{proof}

The second claim implies
\begin{equation}\label{AlmostEquality}
d(o,x)+d(o,y)\ge d(x,y)\ge d(o,x)+d(o,y)-2D.
\end{equation}
Given $o\in [1, x)\cap T$, we denote by $Y_o$ the set of elements $y\in T\setminus \Omega(x)$ such that $o \in[1, x]$ is the branch point of $[1, y]$ and $[1, x]$ in $\mathcal T$. The argument of the first Claim also yields
$$\PS{\Omega(o)}{o} \asymp_L\PS{Y_o}{o}.$$
Then (\ref{BiLipCone}) and (\ref{AlmostEquality}) imply
$$
\sum_{y \in Y_o} \exp(-sd(x, y))   \asymp_{L}   \exp(-sd(o, x))\cdot \PS{T}{1},
$$
 for every $o\in [1, x)\cap T$. By construction of $T$ in Lemma \ref{ITT} the sequence of points $[1, x)\cap T$ has the property that any two consecutive points has a distance between $L-\Delta$ and $L+\Delta$. Summing up over all $o\in [1, x)\cap T$, we get
$$
\sum_{y \in T\setminus \Omega(x)} \exp(-sd(x, y)) \asymp_{L} \sum_{0\le k < d(1, x)} \exp(-sk) \cdot  \PS{T}{1} \asymp_{L} \PS{T}{1},
$$
which proves (\ref{BiLipCone2}). The Lemma is proved.
\end{proof}

\begin{lem}\label{Divergent}
Under the same assumptions in Lemma \ref{ITT}, the Poincar\'e series $\PS{T}{1}$ is divergent at $s=\g T$. Furthermore,
$\displaystyle \lim_{L\to \infty}\g T = \g G$.
\end{lem}
\begin{proof}
It is inspired by the proof of Proposition 4.1 in \cite{DPPS}. Consider the annulus set in $T$,
$$
A_T(g, n, 3\Delta_0) :=A(g, n, 3\Delta_0)\cap T,
$$
where $\Delta_0 :=\Delta + L+2R$ and $n\ge 0$. Observe that there exists $c>0$ such that
\begin{equation}\label{HomoAnnulus}
c^{-1}\cdot \sharp A_T(g', n, 3\Delta_0) \le \sharp A_T(g, n, 3\Delta_0) \le c \cdot \sharp A_T(g', n, 3\Delta_0)
\end{equation}
for any $g, g'\in T$ and $n\ge 0$. Indeed, this is a direct consequence of Lemma \ref{ITT} that $T$ has certain periodicity. Moreover, we claim that
\begin{claim}
The following inequality holds
$$\sharp A_T(1, n+m, 3\Delta_0) \le c \cdot \sharp A_T(1, n, 3\Delta_0) \cdot \sharp A_T(1, m, 3\Delta_0),$$
for $n, m\ge 0$.
\end{claim}
\begin{proof}[Proof of the Claim]
For $h \in A_T(n+m, 3\Delta_0)$, we connect $1$ and $h$ by a geodesic $[1, h]$ in $\mathcal T$. Assume that $d(1, h)=m+n+3\Delta_1$ for some $|\Delta_1| \le \Delta_0$.  Let $z \in [1, h]$ such that $d(1, z)=n+3/2\cdot \Delta_1$. Note that $z$ might not be in $T$. However, by Lemma \ref{ITT}.4), there exists $w\in T$ such that $d(z, w) \le \Delta+L +2R = \Delta_0$ and then $d(w, h) \le m + 3\Delta_0$. This implies that $w\in A(1, n, 3\Delta_0)$ and $h \in A_T(w, m, 3\Delta_0)$. The conclusion thus follows from (\ref{HomoAnnulus}).
\end{proof}

Define $a_n=c\cdot \sharp A_T(1, n, 3\Delta_0)$. The above Claim implies that  $a_{n+m}\le a_n a_m$. So the sequence $(\log a_n)_n$ is  subadditive. Then by  Fekete Lemma   $\displaystyle \lim_{n\to \infty} \frac{\log a_n}{n}=\inf\Bigl\{\frac{\log a_n}{n}: n\ge 1\Bigr\}.$
Since $(a_n)_n$ is non-decreasing we have
  $a_n\leq\sum_{0\le i\le n}a_i\leq na_n$. So

 $$ \g T= \limsup_{n\to \infty} \frac{\log \sum_{0\le i\le n}a_i}{n} =\lim_{n\to \infty} \frac{\log a_n}{n}=\inf\Bigl\{\frac{\log a_n}{n}: n\ge 1\Bigr\}.$$
It follows that $\sharp A_T(1, n, \Delta_0) \ge c^{-1}\exp(n\g T)$ for $n \ge 1$. Observe that
$$
\PS{T}{1} \asymp_{L, \Delta} \sum\limits_{n \ge 0} \sharp A_T(1, n, \Delta_0) \cdot \exp(-sn), \; s \ge 0,
$$
whenever both parts are finite. Thus, $\PS{T}{1}$ is divergent at $s=\g T$.

To prove the second statement we estimate the lower bound of $\g T$. By Lemma \ref{ITT}, we notice that $$\sharp (B(1, i(L+\Delta)) \cap T) \ge \theta^i \cdot \exp(i\cdot \g G\cdot L),$$ for $i\ge 0$. This implies that $$\g T \ge \frac{\log \sharp B_T(1, i(L+\Delta))\cap T}{i(L+\Delta)} \ge \frac{L\cdot \g G+\log \theta}{L+\Delta}.$$
We obtain $\displaystyle \lim_{L\to\infty} \g T\ge\delta_{G,S}.$ Since $\g T\le  \delta_{G,S}\ (\forall  L),$  the lemma follows.
\end{proof}

\subsection{Patterson-Sullivan measures on the space of ends of an iterated transitional tree}\label{Section4.2}
In this and next subsections, for any $L\gg0$, let $T$ and $\mathcal T$ be the iterated tree set and transitional tree respectively given by Lemma \ref{ITT}. At the same time, assume that they  satisfy Lemmas \ref{PoincareSeries} and \ref{Divergent}.

We denote by the common notation $\pT$ the limit set of $T$ in either the Bowditch boundary $\pG$ or in the Floyd boundary $\pGf$. In this subsection, we shall construct a Patterson-Sullivan measure on $\pT$.

Consider the set $\mathcal M(\widetilde T)$ of finite Borel measures on
$\widetilde T:=T \cup \pT$, which is endowed   with the weak-convergence topology.
Then $\mu_n \to \mu$ for $\mu_n \in \mathcal M(\widetilde T)$ if and only if
$\liminf\limits_{n \to \infty}\mu_n(U) \ge \mu(U)$ for any open set
$U \subset  \widetilde T$. Note that a set of uniformly bounded measures in $\mathcal M(\widetilde T)$ is relatively compact.

We first construct a family of measures $\{\mu_v^s\}_{v\in T} \subset \mathcal M(\widetilde T)$ supported on $T$. Set
$$\mu^s_v = \frac{1}{\PS{T}{1}} \sum\limits_{g \in T} \exp(-sd(v, g)) \cdot \dirac{g},$$
where $s >\g T$ and $v \in T$. By Lemma \ref{PoincareSeries}, the measures $\{\mu_v^s\}_{v\in T}$ are bounded by a uniform constant depending on $L$.

By Lemma \ref{Divergent}, for any $v\in T$, $\PS{T}{v}$ is divergent at $s=\g T$. Choose $s_i \to \g T$ such that $\mu_v^{s_i}$ converge in
$\mathcal M(\widetilde T)$. The limit
measures $\mu_v = \lim \mu_v^{s_i}$ are called \textit{Patterson-Sullivan measures}
at $v$. Clearly,
$\{\mu_v\}_{v\in G}$ are absolutely continuous with respect to each other.

In the sequel, we will write PS-measures as shorthand for Patterson-Sullivan
measures.

A horofunction co-cycle $B_\xi: G\times G \to \mathbb R$ at conical points $\xi \in \pG$ or $\xi \in \pGf$ was studied in \cite{YANG7}. The precise definition is not relevant here, but we have the following estimation.

\begin{lem}\label{rephorofunc}\cite[Lemma 2.20]{YANG7}
For any $L>0$ there exists $C=C(L)>0$ such that the following holds.

Fix $\xi \in \partial T$. For any $x, y \in G$, there is
a neighbourhood $V$ of $\xi$ in $\Gf$ or $G \cup \pG$ such that the following property holds:
$$|B_\xi(x, y)-B_z(x, y)| < C, \forall z \in V \cap G,$$
where $B_z(x, y) := d(z, x) - d(z, y)$.
\end{lem}
\begin{rems} (on the proof)
The above statement is proved in \cite[Lemma 2.20]{YANG7} for a conical point of the Bowditch boundary, where the constant $C$ is universal (not depending on $L$).  In our setting  by Lemma \ref{ITT} there exists an $(\epsilon, R, L')$-transitional ray in the tree $\mathcal T$ ending at $\xi$ in $\partial T$. Then by Lemma  \ref{charuc} the constant $R$ is uniform for every $\xi\in\partial T$. So  the same proof as \cite[Lemma 2.20]{YANG7} works to produce a constant $C=C(L)$.

We have to warn the reader that the constant $C>0$ cannot be made  uniform for all conical points for the action  $G\act \pGf$ on the Floyd boundary as the action is not necessarily geometrically finite (see the discussion after Lemma \ref{charuc}).

\noindent{\it End of remarks.}
\end{rems}

 With the help of Lemma \ref{rephorofunc}, the following can be proven exactly as Th\'eor\`eme 5.4 in \cite{Coor}.

\begin{lem}\label{quasiconf}
PS-measures $\{\mu_g\}_{g\in T}$ on $\pT$ satisfy the following property,
\begin{equation}\label{cdensity}
\frac{d\mu_{g}}{d\mu_{h}}(\xi) \asymp_L \exp(-\g T B_\xi (g, h)),
\end{equation}
for $\mu_h$-a.e. points $\xi \in \pT$ and any
$g, h \in T$ .
\end{lem}

\subsection{Shadow Lemma}
We shall establish a shadow lemma for $\{\mu_g\}_{g\in T}$ on $\pT$.
\begin{defn}[Shadow]

The \textit{shadow $\Pi_r(g)$} at $g \in T$ is the set of points
$\xi \in \pT$ such that there exists SOME geodesic $[1, \xi]$ in $\mathcal T$
intersecting $B(g, r)$.
\end{defn}



\begin{lem}[Shadow Lemma]\label{ShadowLem}
There exists $r_0 > 0$ such
that the following holds
$$
\begin{array}{rl}
\exp(-\g T d(1, g)) \prec \mu_1(\Pi_r(g))  \prec_r \exp(-\g T d(1, g))\\
\end{array}
$$
for any $r > r_0$ and $g\in T$.
\end{lem}
\begin{rem}\label{notautologyshadow}
In \cite{YANG7} the Shadow lemma was proved for the whole group $G$. The current lemma describes the shadows of the points $g\in T$ in terms of $\g T$.
\end{rem}
\begin{proof}
By Lemmas \ref{rephorofunc} and \ref{quasiconf}, there exists $C_1=C_1(L), C_2=C_2(L)>0$ such that the
following holds
\begin{equation}\label{denstybound}
C_1  \exp(-\g T d(1, g)) \le \frac{d\mu_1}{d\mu_{g}}(\xi)  \le C_2 \exp(-\g{T}  d(1, g))
\end{equation}
for $\mu_1$-a.e. points $\xi \in \pT$. So in order to estimate $\mu_1(\Pi_r(g))$ we can do it for $\mu_g(\Pi_r(g))$.

\begin{claim}\label{bdcomp}
Given any $\epsilon > 0$, there is a constant $r_0 > 0$ such that
the following holds $$\mu_g(\pT \setminus \Pi_r(g)) < \epsilon$$ for
all $g \in T$ and $r > r_0$.
\end{claim}
\begin{proof}[Proof of the Claim]
Note that $\Pi_r(g)$ is a closed set. We consider the \textit{convex cone} $\mathcal C(\pT \setminus \Pi_r(g))$ of $\pT \setminus \Pi_r(g)$, which consists of all geodesic rays in $\mathcal T$originating at $1$ and terminating at a point in  $\pT \setminus \Pi_r(g)$. Let $V$ be the set of vertices of $T$ in $\mathcal C(\pT \setminus \Pi_r(g))$.

 For any $x\in V$, consider the branch point $o$ of $[1, x]$ and $[1, g]$ in $\mathcal T$ (defined in the proof of Lemma \ref{PoincareSeries}). Since $x \notin \mathcal C(\Pi_r(g))$, we have $d(g, o) > r$. By a similar argument to that of
   Lemma \ref{PoincareSeries}, we get $$\PS{V}{g} = \sum\limits_{x \in V} \exp(-sd(x, g)) \asymp \sum_{r \le k \le d(1, g)} \exp(-sk) \cdot  \PS{T}{1}.$$

 So,
$$\mu_g^s(V)=\frac{\PS{V}{g}}{ \PS{T}{1}} \asymp \sum_{r \le k \le d(1, g)} \exp(-sk) ,$$
which tends to $0$ when $r \to \infty$ and   $s> \g T$.

Thus, the $\mu_g^s$-measure of the open set $V\cup (\pT \setminus \Pi_r(g))$ can be arbitrarily small  for $r$ large enough, and so is $\mu_g(\pT \setminus \Pi_r(g))$. This proves the claim.
\end{proof}

By Lemma \ref{PoincareSeries}, we have that $\{\mu_g(\pT)\}_{g\in T}$ are  lower and upper bounded by a uniform constant depending on $L$.  Let $\eta_1=1/2\inf\{\mu_g(\pT): g\in T\}>0$ and $\eta_2=\sup\{\mu_g(\pT): g\in T\} < \infty$. By the above Claim, there is a constant
$r_0>0$ such that the following holds
\begin{equation}\label{measurelbound}
\eta_1< \mu_{g}( \Pi_r(g)) < \eta_2,\; \forall r > r_0,
\end{equation}
for all  $g \in T$.   So (\ref{denstybound}) implies that
$$
\eta_1 C_1 \exp(-\g T d(1, g)) \le \mu_1(\Pi_r(g)) \le \eta_2 C_2 \exp(-\g T d(1, g)),
$$
for all $g \in G$. The Lemma is proved.
\end{proof}

By  Shadow Lemma to estimate the PS-measure of balls  we need to compare it with that of  shadows a in the boundary $\pT$. Below, we use the symbol $\floor{s}$ to denote the integer part of $s \in \mathbb R$. Denote by $B_{\fds_{\lambda, 1}}(\xi, t)$ (resp. $B_{\fds_{\lambda, 1}}(\xi, t)$) the  ball in $\pT$ around $\xi \in \pT$ of radius $t$ with respect to the metric $\fds_{\lambda, 1}$(resp. $\bar \fds_{\lambda, 1}$).

\begin{lem}[Shadows $\asymp$ Balls]\label{ShadowDisk}
Let $r_0$ given by Lemma \ref{ShadowLem}. There exists $0 < \lambda_0 <1$ such that for any $\lambda \in [\lambda_0, 1)$ and $L>0$, there exist $r=r(L, \lambda)>r_0$ and $C=C(L,\lambda)>0$ with the following property.

For any $\xi \in \pT$ and $0<t<\lambda$, the following holds
\begin{equation}\label{FloydVisualEQ}
B_{\fds_{\lambda, 1}}(\xi, C^{-1}t)
\subset \Pi_r(g) \subset B_{\fds_{\lambda, 1}}(\xi, Ct),
\end{equation}
and
\begin{equation}\label{ShortcutVisualEQ}
B_{\bar \fds_{\lambda, 1}}(\xi, C^{-1}t)
\subset \Pi_r(g) \subset B_{\bar\fds_{\lambda, 1}}(\xi, Ct),
\end{equation}
where $g\in [1, \xi]$ is chosen such that $d(1, g)= \floor{\log_\lambda t}$.
\end{lem}
\begin{proof} Let $\lambda_0$ be given by Proposition \ref{Floydmap}. For any $\lambda \in [\lambda_0, 1[$, we consider the family of Floyd metrics $\{\fds_{\lambda, v}\}_{v_\in G}$.

For any $0<t<\lambda$, let $g\in [1, \xi]$ such that $d(1, g)= \floor{\log_\lambda t}$. Thus,  $$\lambda^{d(1, g)+1} < t \le \lambda^{d(1, g)}.$$ By   construction of $T$ (see Lemma \ref{ITT}.5), we know that $[1, \xi]$ is $(\epsilon, R, L)$-transitional.  So there exists an $(\epsilon, R)$-transitional point $z$ in $[1, \xi]$ such that $d(z, g)\le L$.  By Lemma \ref{floydtrans}.2, there exists $\kappa=\kappa(\lambda, \epsilon, R)$ such that $\bar \fds_{\lambda, z}(1, \xi)\ge \kappa$. By
 property (\ref{lambdabilip}), we have $\fds_{\lambda, g}(1, \xi)\ge \bar\fds_{\lambda, g}(1, \xi)\ge \kappa\cdot \lambda^L$.

Set $2C_1=\kappa\cdot \lambda^{L}$  and $r=\max\{\phi_\lambda(C_1), r_0\}$ where the function $\phi$ is given in Lemma \ref{karlssonlem}. Let $\eta \in B_{\fds_{\lambda, 1}}(\xi, C_1 t)$. By
property (\ref{lambdabilip}), it follows that $\fds_{\lambda, g}(\eta, \xi)
\le \lambda^{-d(g,1)}\fds_{\lambda, 1}(\eta, \xi) \le C_1$. Then $\fds_{\lambda, g}(1, \eta) \ge  C_1$ and by  Lemma \ref{karlssonlem}, we have $d(g, [1, \eta]) \le r$. So $\eta \in \Pi_r(g)$. This proves the first inclusions of (\ref{FloydVisualEQ}) and (\ref{ShortcutVisualEQ}) for $C=C_1$

Let $\eta \in \Pi_{r}(g)$ so that $d(g, [1, \eta]) \le r$ for some geodesic $[1, \eta]$. Consequently, there exists $w\in [1, \eta[$ such that $d(1, w)=d(1, g)$ and $d(g, w) \le 2r$. By Lemma \ref{Floydgeod} any segment of $[1, \xi]$ is a Floyd geodesic with respect to $\fds_{\lambda, 1},$ so $\fds_{\lambda, 1}(\xi, g)={\lambda^{d(1,g)}\over 1-\lambda}.$ Let $\alpha$ be  a word geodesic  between $w$ and $g$. Every edge of $\alpha$ is in the word distance at most $d(1,g)-2r$ from $1.$ So the Floyd length of $\alpha$ is at most
$2r\cdot \lambda^{d(1,g)-2r}$.
We  obtain
$$
\begin{array}{lll}
\bar \fds_{\lambda, 1}(\xi, \eta) \le \fds_{\lambda, 1}(\xi, \eta) &\le \fds_{\lambda, 1}(g, \xi)+\fds_{\lambda, 1}(w, \eta) + \fds_{\lambda, 1}(g, w) \\
\\
&\le 2(\frac{1}{1-\lambda}+\frac{r}{\lambda^{2r}})\cdot {\lambda^{d(g,1)}}.
\end{array}
$$
Let $\displaystyle C_2=2\lambda^{-1}(\frac{1}{1-\lambda}+\frac{r}{\lambda^{2r}})$. Then $\bar\fds_{\lambda, 1}(\xi, \eta) \le  \fds_{\lambda, 1}(\xi, \eta)\le C_2 t$ and so the second inclusions of (\ref{FloydVisualEQ}) and (\ref{ShortcutVisualEQ}) follow.

Setting $C=\max\{C_1, C_2\}$ we complete the proof of the lemma.
\end{proof}

\subsection{Proof of Theorem \ref{ThmAR}}
We recapitulate the main results of the previous Subsections in the following.
\begin{prop}\label{PSmeasureonT}
There exists $\lambda_0>0$ such that for every $\lambda \in [\lambda_0, 1)$ and  $L\gg 0$, there exist an $L$-iterated tree set $T$ and a PS-measure $\mu_1$ on $\pT$ satisfying:
\begin{equation}\label{measureball}
\mu_1(B_{\fdo}(\xi, t)) \asymp_{\lambda, L} t^{-\g T/\log \lambda},
\end{equation}
for any $\xi \in \pT$ and $0<t <\lambda$.
\end{prop}

\begin{proof} The existence of the tree $T$ is proved in Lemma \ref{ITT}. Lemmas \ref{ShadowLem} and \ref{ShadowDisk} and direct calculations imply that  $\pT$ is Alhfors $Q$-regular  for $Q=-\g T/\log \lambda$ (see the definition in the Introduction). Hence, (\ref{measureball}) follows.
\end{proof}

 By Lemma \ref{charuc} $\pT$ consists of uniformly conical points, so Proposition \ref{PSmeasureonT} implies the first claim of the Theorem. The statement $Q_i \to \g G\ (i \to \infty)$ is proved in  Lemma \ref{Divergent}. Theorem  \ref{ThmAR} is proved.

\section{Proofs of Theorems \ref{ThmHD} and \ref{ThmHD2}}\label{Section5}
We consider the Floyd metric on $\pGf$ and shortcut metric on $\pG$, where the corresponding Theorems \ref{ThmHD} and \ref{ThmHD2} are proved with the same argument.

The following lemma giving the upper bound for the Hausdorff dimension is due to Marc Bourdon. We notice that it is a general fact which is true for a finitely generated group $G$ without assuming that it is relatively hyperbolic.
\begin{lem}[M. Bourdon, oral communication]\label{Bourdon}
For every $\lambda \in (0, 1)$, the Hausdorff dimension $\HD_{\fds_{\lambda, 1}}$ of $\pGf$ (respectively $\HD_{\bar \fds_{\lambda, 1}}$ of $\pG$) with respect to the Floyd metric $\fds_{\lambda, 1}$ (respectively to the shortcut metric $\bar \fds_{\lambda, 1}$ is upper bounded by $-\g G/\log \lambda$.
\end{lem}
\begin{proof}
To give an upper bound, it suffices to prove that $\mathcal H^s(\pGf)=0$ for any fixed $s > -\g G/\log \lambda$.

Define $S_n=\{g\in G: d(1, g)=n\}$. For any $g \in S_n$, define the cone $\Omega_g :=\{\xi \in \pGf, g\in [1, \xi]\}$, where $[1, \xi]$ is a geodesic between $1$ and $\xi$.

For any $\xi\in \pGf$, consider a point $x \in [1, \xi]\cap S_{n}$. By Lemma \ref{floydrays} the sub-ray $[x, \xi)$ is a $\fds_{\lambda, 1}$-Floyd geodesic. So   $\fds_{\lambda, 1}(x, \xi) = \frac{\lambda^n}{1-\lambda}$ for any $\xi\in \pGf$. Thus, $\{\Omega_g: g\in S_n\}$ is an $\varepsilon$-covering of $\pGf$, where $\varepsilon:=\frac{2\lambda^n}{1-\lambda}$.

For any $t \in ]-\frac{\g G}{\log \lambda}, s[$, we have $-t\log \lambda >\g G$ and so $\sharp S_n \prec_t \lambda^{-tn}$ for $n\ge 1$. We obtain for all $n \ge 1$:
$$
\mathcal H^s(\pGf) \le \sum_{g\in S_n} \varepsilon^s \prec  \lambda^{(s-t)n}
$$
which then tends $0$ as $n\to \infty$. Thus, $\mathcal H^s(\pGf)=0$ for any $s > -\frac{\g G}{\log \lambda}$. The lemma is proved.
\end{proof}

So the upper bound on the Hausdorff dimension of $\pGf$ and $\pG$ in Theorems \ref{ThmHD}, \ref{ThmHD2} is proved. In the remainder of proofs, we aim to establish the lower bound for the Hausdorff dimension.

Taking into account Proposition \ref{PSmeasureonT}: there exists a universal $\lambda_0>0$ such that for each $L\gg 0$, there exist an $L$-iterated tree $T$ and a PS-measure $\mu_1$ on $\pT$ such that (\ref{measureball}) holds and $\g T \to \g G$ as $L \to \infty$.

The following lemma shows that PS-measures
constructed in Section \ref{Section4} are actually the Hausdorff measures
on $\pT$ with respect to the  Floyd metric $\fds_{\lambda, 1}$ restricted on $\pT$.

\begin{lem}\label{Unique}
Let $\mu_1$ be a PS-measure on $\pT$ in $\pG$ or $\pGf$. Denote $\sigma=-\g
T/\log \lambda$. Then we have
$$
\mathcal H_{\sigma}(A) \asymp_L \mu_1(A).
$$
for any subset $A \subset \pT$.
\end{lem}
\begin{proof}
In the proof, we assume that $\pT$ is a subset of  the Bowditch boundary. The proof  for $\pT\subset \pGf$ is similar.

Let $\mathcal B$ be an
$\varepsilon$-covering of $A$ for $\varepsilon>0$. Then $\mu_1(A) \le \sum_{B \in \mathcal B}
\mu_1(B)$. Let $\varepsilon \to 0$. By Proposition \ref{PSmeasureonT}, we obtain
that $\mu_1(A) \prec_L \mathcal H_{\sigma}(A)$.

For the other inequality, we need to make use of the following well-known covering result. Let $B$ be a metric ball of radius $\textrm{rad}(B)$ in a proper metric space $X$. Denote by $5B$ the union of all balls of radius $2\cdot \textrm{rad}(B)$ intersecting $B$ so that $\diam{5B}\le 10\cdot \textrm{rad}(B)$. Then  by  \cite[Theorem 2.1]{Matt} for a family of balls $\mathcal B$
in $X$ with uniformly bounded radii there exists a sub-family
$\mathcal B' \subset \mathcal B$ of pairwise disjoint balls such
that the following holds
\begin{equation}\label{vatlicovering}
\bigcup\limits_{B \in \mathcal B} B \subset \bigcup\limits_{B \in
\mathcal B'} 5B.
\end{equation}

Note that $\mu_1,
\mathcal H_{\sigma}$ are Radon measures. Then for any $\tau
>0$ there exists a compact set $K$ and an open set $U$ such that $K
\subset A \subset U$ and $\mathcal H_{\sigma}(U\setminus K) < \tau, \mu_1(U
\setminus K) < \tau$.

Set $\epsilon_0:=\fds_1(K, \pG \setminus U) >0$. For any $0 < \epsilon
< \epsilon_0$, let $\mathcal B$ be an $\varepsilon$-covering of $K$.
By (\ref{vatlicovering}) and Proposition \ref{PSmeasureonT}, there exists a sub-family $\mathcal B'$ of $\mathcal B$ such that
$$
H_{\sigma}(K)\le \sum_{B
\in \mathcal B'} (\diam{5B})^{\sigma}  \le  \sum_{B
\in \mathcal B'} (10 \cdot \textrm{rad}(B))^{\sigma} \prec_L \mu_1(U).
$$
The condition
$\tau\to 0$ yields $ H_{\sigma}(A) \prec \mu_1(A)$.
\end{proof}
\begin{rem}
We note that  $\mu_1$ is unique in
the following sense: if $\mu_1, \mu'_1$ are two PS-measures, then $d\mu_1/d\mu'_1$ is bounded from up and below.
\end{rem}

Lemma \ref{Unique} proves that the Hausdorff dimension of $\pT$ is equal to $\sigma$. Since $\pT$ is a subset of the set of uniformly conical points in $\pGf$ and $\pG$, the dimension $\sigma=-\g
T/\log \lambda$ of $\pT$ gives a lower bound of $\HD_{\fds_{\lambda, 1}}(\uGf)$ and $\HD_{\bar \fds_{\lambda, 1}}(\uG)$.

Letting $L\to\infty$, we have $\g T \to \g G$ by Proposition \ref{PSmeasureonT}. So, $$\HD_{\fds_{\lambda, 1}}(\uGf)\ge -\g
G/\log \lambda$$ and $$\HD_{\fds_{\lambda, 1}}(\uG)\ge -\g
G/\log \lambda.$$ The proofs of Theorems \ref{ThmHD} and \ref{ThmHD2} are complete.

\section{Tight Paths and Floyd metrics}\label{Section3}

In this section, we shall develop a detailed understanding of shortcut geodesics via a class of well-controlled paths called (generalized) tight paths.

\subsection{Tight paths}
It is well-known that in hyperbolic spaces, a sufficiently ``long'' local geodesic becomes globally a quasi-geodesic. This property  in general fails for the Cayley graph of a relatively hyperbolic group. V. Gerasimov and L. Potyagailo proposed  in \cite{GePo3} a notion of tight paths as a generalization of local geodesics to the relative setting. The following definition is a small modification of it.

\begin{defn}\label{tightdefn}
For $c\ge 1, l>0$, a path $\gamma$ is called \textit{$(c, l)$-tight} path if for any two points $x, y\in \gamma$ with $d(x,y)\le l$ the subpath $[x, y]_\gamma$  is a $c$-quasi-geodesic.
\end{defn}
\begin{rem}
This definition is a partial case of \cite[Definition 6.1]{GePo3} where a local quasi-geodesicity is requested outside of the horospheres only and  an additional condition is assumed for the horospheres. So if a path is tight in the sense of Definition \ref{tightdefn} it is also tight in the sense of   \cite[Definition 6.1]{GePo3} but not necessarily vice versa. In particular we can use all results proven in \cite{GePo3}. In addition, the above definition implies that every subpath of a tight path is   a tight
path itself, which is not always true in the general case. This stability of the tightness  for subpaths will be often used below.

We also stress that the above definition does not coincide with the standard notion of local (quasi-)geodesicity when the  assumption   that the length of a subpath (and not its diameter) is small implies   its (quasi-)geodesicity.
\end{rem}

In what follows, to reduce cumbersome quantifiers, we continue to use Convention \ref{epsilonR1} without explicit mention on the constants $\epsilon, R$, which depend on the parameter $c>0$ in tight paths.

We recall the following result about tight paths proved in \cite{GePo3}.

\begin{lem}\label{tightpath}
For any $c\ge 1$, there exist $\kappa=\kappa(c), l_0=l_0(c)>0$ with the following property.

Let $\gamma$ be a $(c, l)$-tight path for $l\ge l_0$. Then $\fds_{v}(\gamma_-,\gamma_+) \ge \kappa$ for any $(\epsilon, R)$-transitional
point $v\in \gamma$.
\end{lem}
\begin{proof}[Comments on the proof]
The statement that $\fds_{v}(\gamma_-,\gamma_+) \ge \kappa$ is first established in \cite[Proposition 6.7]{GePo3} for a special sub-sequence of transitional vertices $v\in\gamma.$  Then it is shown in  the proof of \cite[Theorem B]{GePo3} that  the tigthtness of a path implies it for every transitional vertex (up to decreasing the constant $\kappa$).
\end{proof}

We  call below  a sequence of   points $z_i=\gamma(t_i)$ of a length-parametrized path $\gamma$ {\it well-ordered} if $t_i > t_{i-1}\ (t_i\in\mathbb Z).$

The following lemma is an intermediate step in the proof  of   Proposition \ref{tightadmissible} below which is the main result of this subsection.

\begin{lem}[Transitional tight path is quasi-geodesic]\label{tightranspath}
For any $c, L\ge0$, there exist $l_0=l_0(L), c'=c'(c)\ge 1$ with the following property.

Let $\gamma$ be a $(c, l)$-tight path for $l\ge l_0$. Assume that $\gamma$ is an $(\epsilon, R, L)$-transitional path, where $\epsilon, R$ satisfy Convention \ref{epsilonR1}. Then $\gamma$ is a $c'$-quasi-geodesic.
\end{lem}
\begin{proof}
By Lemma \ref{tightpath}, there exists $\kappa=\kappa(c)\ge 0$ such that $\fds_{x}(\gamma_-,\gamma_+) \ge \kappa$ for every $(\epsilon, R)$-transitional point $x \in \gamma$. Set $D_0=\phi(\kappa/2)$. Choose $l_0\ge 2(L+D_0)$.

Since any subpath of $\gamma$ is $(c, l)$-tight, it is enough to prove that there exists a linear bound for $\len(\gamma)$ with respect to $d(\gamma_-, \gamma_+)$. Let $\alpha$ be a geodesic with the same endpoints as $\gamma$. The idea of proof is to find two sequences of well-ordered points in $\gamma$ and $\alpha$ respectively which are uniformly close.

Since $\gamma$ is $(\epsilon, R, L)$-transitional, there exists a maximal set of $(\epsilon, R)$-transitional well-ordered points $\{z_i: 1\le i \le n\}$ in $\gamma$ such that $$\len([z_i, z_j]_\gamma)\ge 2D_0$$ for $i\ne j$ and $$\len([z_i, z_{i+1}]_\gamma)\le 2(L+D_0)$$ for $1 \le i < n$. Indeed, let $z_1$ be the first $(\epsilon, R)$-transitional point in $\gamma$.  Suppose $z_i$ is chosen for $i\ge 1$. If $\len([z_i, \gamma_+])\leq 2(L+D_0)$ then  $z_{i+1}=\gamma_+$. Consider the point $z$ in $[z_i, \gamma_+]_\gamma$ such that $\len([z_i, z]_\gamma) = L+2D_0$. If $z$ is $(\epsilon, R)$-transitional in $\gamma$, then set $z_{i+1}=z$. Otherwise, there exists an $(\epsilon, R)$-transitional point $z_{i+1}$ such that $\len([z, z_{i+1}]_\gamma) \le L$  and $\len([z_i, z_{i+1}]_\gamma) \le 2(L+D_0)$.

By  Lemma \ref{tightpath}, there exists $\kappa>0$ such that
$$\fds_{z_{i+1}}(z_i, \gamma_+) \ge \kappa$$
for any $1\le i < n$. By Lemma \ref{karlssonlem}, there exists $w_1\in \alpha$ such that $d(z_1, w_1) \le D_0$. We now choose other $w_i$ inductively for $i\ge 1$.

Suppose $w_i \in \alpha$ is chosen such that $d(z_i, w_i)\le D_0$. Since $d(z_{i+1}, z_i) \ge 2D_0$, we obtain $[z_i, w_i]\cap B(z_{i+1}, D_0)=\emptyset$. By the choice of $D_0=\phi(\kappa/2)$, we know that for any $v\in G$, any geodesic outside $B(v, D_0)$ has $\fls_v$-length at most $\kappa/2$. So $\fds_{z_{i+1}}(w_i, z_i) \le \kappa/2$ and then $\fds_{z_{i+1}}(w_i, \gamma_+) \ge \kappa/2$.  Thus there exists $w_{i+1} \in [w_i, \alpha_+]_\alpha$ such that $d(z_{i+1}, w_{i+1})\le D_0$. Clearly, the obtained points $w_i$ are well-ordered on $\alpha$.

As $l_0>2(L+D_0)$,   $[z_i, z_{i+1}]_\gamma$ is a $c$-quasi-geodesic by the tightness property. Since $w_i$ are well-ordered on $\alpha$, we see that $\gamma$ is a $c'$-quasi-geodesic for $c':=c+2D_0$.
\end{proof}

The following lemma will be often used further.

\begin{lem}[Bounded overlap]\label{bddoverlap}
For $c\ge 1$ and $(\epsilon, R)$ given by Convention \ref{epsilonR1}, there exist $K_0, l_0>0$ with the following property.

Let $\gamma$ be a $(c, l)$-tight path for $l \ge l_0$.  Assume that $\beta_1, \beta_2$ are two maximal connected segments of $\gamma$ such that $(\beta_i)_-, (\beta_i)_+ \in N_\epsilon(X_i)$ for some $X_i \in \mathbb P$ with $i=1,2$. Then $\len(\beta_1\cap \beta_2) \le K_0$. In particular, the endpoints of $\beta_i$ are $(\epsilon, R)$-transitional for $i=1, 2$.
\end{lem}
\begin{proof}
By Definition \ref{tightdefn} a subpath of a tight path is itself tight. Then by \cite[Proposition 7.6]{GePo3} it follows that there exists $l_0>0$ such that for all $l\geq l_0$ the elements of $\mathbb P$ are uniformly quasi-convex with respect to the system of $(c,l)$-tight paths.  This implies that there exists a uniform constant $\varepsilon=\varepsilon(\epsilon, c)>0$ such that $\beta_i \subset N_\varepsilon(X_i)$ for $i=1, 2$.   By Lemma \ref{periphsystem} we find  a constant $R=\mathcal R(\varepsilon) >0$ such that  $\diam{N_{\varepsilon}(X)\cap N_{\varepsilon}(X')}\leq  R$  for every $X, X'\in \mathbb P.$ Assume that $l_0 > { c R+c}$.  Since $\beta_i$ are $l$-local $c$-quasi-geodesic for $l >l_0$, it follows that $\len(\beta_1\cap \beta_2) \le K_0 :={c R+c}$.
\end{proof}
\begin{rem}
By the bounded intersection of $\mathbb P$, this lemma holds trivially if $\gamma$ is a quasi-geodesic. However, the tight path $\gamma$ above is    a local quasi-geodesic only.
\end{rem}

Let $\gamma$ be a $(c, l)$-tight path. Let $\epsilon=\epsilon(c)$ given by Convention \ref{epsilonR1} and $K_0$ given by Lemma \ref{bddoverlap}. For $K >K_0$, we consider all maximal connected segments $\beta_i$ in $\gamma$ ($1\le i \le m$) such that $\len(\beta_i) \ge K$ and $(\beta_i)_-, (\beta_i)_+ \in N_\epsilon(X_i)$ for some $X_i \in \mathbb P$. Consequently, $X_i\ne X_j$ for $i\ne j$.  These $(\beta_i, X_i)$ shall be referred to as \textit{$(\epsilon, K)$-components} of $\gamma$.

We stress that by the argument of Lemma \ref{bddoverlap}   the segment  $\beta_i$ belongs to $N_\epsilon(X_i)$ for a uniform $\epsilon>0$ and unique $X_i.$

We now introduce a modification of a tight path to make the obtained path a quasi-geodesic.

\begin{defn}[Truncation of  a tight path]\label{truncpath}
Let $\gamma$ be a $(c, l)$-tight path for $c\ge 1, l>0$. Consider all $(\epsilon, K)$-components $(\beta_i, X_i)$ ($1\le i \le m$) for a fixed $K  > 2K_0$, where $K_0>0$ is given by Lemma \ref{bddoverlap}.

Set $y_1=(\beta_1)_-, x_2=(\beta_1)_+$. If $\beta_i\cap \beta_{i-1}=\emptyset$ for $i\ge 2$, denote $y_i=(\beta_i)_-, x_{i+1}=(\beta_i)_+$; otherwise, set $y_i=x_{i-1}, x_{i+1}=(\beta_i)_+$. Replace $[y_i, x_{i+1}]_\gamma$ by a geodesic segment $[y_i, x_{i+1}]$ for each $i\ge 1$.

The path $\bar \gamma$ obtained in this way is called a \textit{$K$-truncation} of $\gamma$.
\end{defn}
\begin{rem}
The following observation is elementary and useful: every $\beta_i$ produces an $(\epsilon, K/2)$-deep point in $X_i$ in the truncation path $\bar \gamma$. Consequently, if $\bar \gamma$ does not contain an $(\epsilon, R)$-deep point, then $d((\beta_i)_-, (\beta_i)_+) \le 2R$ for all $\beta_i$.
\end{rem}

The following lemma is the main result of this subsection. It provides a  further generalization of Lemma \ref{tightranspath} to the truncated tight paths.

\begin{lem}[Truncation is quasi-geodesic]\label{tightadmissible}
For any $c\ge 1$, there exist $l_0=l_0(c), K=K(c), c'=c'(c)>0$ with the following property. For any $l\ge l_0$, the $K$-truncation of a $(c, l)$-tight path is a $c'$-quasi-geodesic.
\end{lem}
\begin{proof}
Let $K>2K_0$ be a fixed integer, where $K_0$ is given by Lemma \ref{bddoverlap}. Let $\bar\gamma$ be the $K$-truncation of a $(c, l)$-tight path $\gamma$.  Keeping the notations as in Definition \ref{truncpath}, we have by Lemma \ref{bddoverlap}   that $y_i, x_{i+1}$ for $1\le i\le m$ are $(\epsilon, R)$-transitional points in $\gamma.$ Furthermore since $[x_i, y_{i}]_\gamma$ contains no $(\epsilon, K)$-components for $1\le i< m$, we see that $[x_i, y_{i}]_\gamma$ is an $(\epsilon, R, L)$-transitional path for $L:=K/2$. By Lemma \ref{tightranspath}, there exist $l_0=l_0(L), c_0=c_0(c)\ge 1$ such that $[x_i, y_{i}]_\gamma$ is a $c_0$-quasi-geodesic.

The geodesic $[y_i, x_{i+1}]$ belongs to the $\epsilon$-neighbourhood $N_\epsilon(X_i)$ for some $X_i\in \mathbb P $ where we assume that $y_i$ is the  entry point of $\gamma\cap N_\epsilon(X)$ and $x_{i+1}$ is the exit point of it. We will now show that $\gamma$ is a quasi-geodesic in a neighbourhood of $y_i.$

Since $y_i$ is $(\epsilon, R)$-transitional,   $\fds_{y_i}(x_i,x_{i+1}) \ge \kappa$ where $\kappa$ is given by Lemma \ref{tightpath}.
Hence, $d(y_i, [x_i,x_{i+1}])\leq D_0:=\phi(\kappa)$ where $\phi$ is the function given by  Lemma \ref{karlssonlem}. Connect the point $x_i$ with an arbitrary point $z\in [y_i, x_{i+1}]$ by a geodesic $\beta$. By the triangle inequality we have $d(x_i,y_i)+d(y_i,x_{i+1})\leq d(x_i,x_{i+1})+2D_0$. Since $d(y_i,z)+d(z,x_{i+1})=d(y_i,x_{i+1})$ we obtain $d(x_i, y_i)+d(y_i,z)\leq d(x_i,x_{i+1})-d(z,x_{i+1})+2D_0\leq d(x_i,z)+2D_0.$
Finally
$$
\begin{array}{lll}
\len([x_i,z]_{\bar\gamma})&=\len([x_i,y_i]_\gamma)+\len([y_i,z])\\
&\leq c_0d(x_i,y_i)+c_0+d(y_i,z)\\
&\leq c_0d(x_i,z)+c_0+2c_0D_0.
\end{array}
$$
So $[x_i,z]_{\bar\gamma}$ is a $c_1$-quasi-geodesic for $c_1:=c_0(2D_0+1).$

We have that any subpath $\tilde \gamma$ of the truncated path $\bar\gamma$ is the union of three types of $c_1$-quasi-geodesic subpaths:   a) $\gamma_i=[x_{i}, y_i]_\gamma$, b) $\beta_i=[y_i, x_{i+1}]_\gamma$ and c) $\delta=[a,b].$ Both vertices of the intervals of types a) and b) are transitional on the corresponding tight path $\gamma$, and   every $\gamma_i$ is $(\epsilon, R, L)$-transitional whereas $\beta_i$ is an $(\epsilon, K)$-component. The path $\tilde\gamma$ can contain at most two intervals $\delta$ of type $c)$ such that one of the endpoints of $\delta$ coincides with an endpoint of $\tilde\gamma$ and is an interior point of a geodesic truncation  of $\bar \gamma.$

Repeating the argument of Lemma \ref{tightranspath} consider a maximal well-ordered subset $V$ of the transitional vertices $\{v_j \in \tilde \gamma\}$ in the set $W:=\tilde \gamma\cap \{y_i, x_{i+1}: 1\le i\le m\}$  such that $d(v_j, v_{j+1})\ge 2D_0$. We connect the endpoints of $\tilde \gamma$ by a geodesic $\alpha.$ Then for each $v_j\in V$, there exists $v_j'\in \alpha$   such that $d(v_j, v_j')\le D_0$ and $v'_j\in [v'_{j-1}, v'_{j+1}]_\alpha$.  Since $V$ is maximal in $W$,  for any $w\in W$ there exists $v\in V$ such that $d(v, w)\le 2D_0$.  If $W=\{w_1, w_2, \cdots, w_n\}$, there exists a well-ordered set $W'=\{w_1', w_2', \cdots, w_n'\}$ of vertices in $\alpha$ such that $d(w_i, w_i')\le 3D_0$. Then $[\tilde \gamma_-,w_1]_{\tilde\gamma}$, $[w_i, w_{i+1}]_\gamma$ and $[w_n, \tilde \gamma_+]_{\tilde \gamma}$ are all $c_1$-quasi-geodesics by the above argument. We have
$$
\begin{array}{lll}
\len(\tilde \gamma)&\leq c_1(d(\tilde \gamma_-,w_1)+\sum_{i=1}^{n-1}d(w_i,w_{i+1}) +d(w_n, \tilde \gamma_+))+c_1\\
&\leq c_1(3D_0+d(\tilde \gamma_-,w'_1)+\sum_{i=1}^{n-1}(d(w'_i,w'_{i+1}) +3D_0)+d(w'_n, \tilde \gamma_+)+3D_0)+c_1\\
&\leq
 c'd(\tilde\gamma_-,\tilde\gamma_+)+c',
 \end{array}
$$
 where $c':=(1+3D_0)c_1.$ The Lemma is proved.
\end{proof}

\begin{conv}\label{lK0}
For any $c \ge 1$, we will assume further on that $l_0, K>0$ satisfy both Lemmas \ref{tightpath} and \ref{bddoverlap}.
\end{conv}

\subsection{Shortcut metrics and generalized tight paths}
Recall that a Floyd geodesic in the Floyd completion does not in general belongs to the Cayley graph and the shortening procedure described in subsection \ref{Section2.3} allows one to approximate them by local geodesics in the graph. Furthermore the following lemma shows that this approximation can be done using the tight paths:

\begin{lem}\cite[Corollary 7.8]{GePo3} \label{tithapprox}
For any $l>0$ there exists $\lambda_0\in]0,1[$ such that for every $\lambda\in]\lambda_0, 1[$ if the Floyd geodesic $\gamma\subset\Gf$ (with respect to the metric $\fdo$) joining  two distinct points $x,y$ in $\Gf$ does not belong to the Cayley graph $\Gx,$ then for $\varepsilon>0$ there exists a tight path $\tilde\gamma\subset\Gx $ such that $\vert\flo(\tilde\gamma)-\flo(\gamma)\vert <\varepsilon.$
\end{lem}

The goal of this subsection is to extend this result to the geodesics with respect to the shortcut metrics $\{\sfdo\}_{o\in G}$ on $\pG$ (see Section \ref{Section2.2}). For this purpose we generalize the notion of a tight path as follows.

\begin{defn}[Generalized tight paths and truncations]\label{gentightdefn}
Let $\gamma$ be a finite sequence of $(c, l)$-tight paths $\gamma_i$ in $\Gx$ ($1 \le i \le  n$) such that $(\gamma_i)_+, (\gamma_{i+1})_- \in N_\epsilon(X_i)$ for some $X_i \in \mathbb P$ where $X_i\ne X_j$ ($1 \le i \ne j <  n$).

We say that $\gamma$ is a $(c, l)$-\textit{generalized tight path} if  for each pair of entry and exit points $y_i, x_{i+1}$ of $\gamma_i$ and $\gamma_{i+1}$ respectively in $N_\epsilon(X_i)$ we have $d(y_i, x_{i+1}) \ge l$  ($1\le i <n$).

Fix $K>0$. For $n>1$, consider the $K$-truncation $\bar \gamma_i$ of $[x_i, y_i]_{\gamma_i}$ where $1\le i\le n$. The path $$\tilde \gamma=\bar \gamma_1\cdot [y_1, x_2]\cdot \bar \gamma_2 \cdots [y_{n-1}, x_n] \cdot \bar \gamma_n$$ is called the \textit{$K$-truncation} of a generalized $(c, l)$-tight $\gamma$.
\end{defn}
\begin{rem}
Note that a generalized tight path is possibly not connected. If it is connected, then it is a tight path in Definition \ref{tightdefn}.
\end{rem}

\begin{lem}[Generalized truncation is quasi-geodesic]\label{jointpath}
For any $c\ge 1$, there exist $l_0, K, c'\ge 1$ such that for any $l>l_0$, the $K$-truncation of a $(c, l)$-generalized tight path is a $c'$-quasi-geodesic.
\end{lem}
\begin{proof}
Let $K=K(c)$ be given by Lemma \ref{tightadmissible}. Let $\tilde \gamma$ be the $K$-truncation of a generalized $(c, l)$-tight path $\gamma$. We keep the notations of Definition \ref{gentightdefn}. If $n=1$, the proof is finished by Lemma \ref{tightadmissible}. Assume that $n\ge 2$.

By Lemma \ref{tightadmissible}, there exists $c_1>0$ such that each $\bar \gamma_i$ is a $c_1$-quasi-geodesic for each $1\le i<n$. We prove below that $\bar \gamma_i$ and $\bar \gamma_{i+1}$ have bounded projection to $N_\epsilon(X_i)$ where $X_i \in \mathbb P$.

By Lemma \ref{periphsystem},  $X \in \mathbb P$ is $c_1$-contracting and there exist $\mu_{c_1}, D_{c_1}>0$ such that the (\ref{contractingproperty}) holds. By Convention \ref{epsilonR1}, we have $\epsilon \ge \mu_{c_1}$.  Let $z$ be the entry point of $\bar \gamma_i$ in $N_\epsilon(X_i)$.

\begin{claim}
There exists a constant $C>0$ such that $d(z, y_i)\leq C$.
\end{claim}
\begin{proof}[Proof of Claim]
Since $y_i$ is the entry point of $\gamma_i$ in $N_\epsilon(X_i)$, there exists an $(\epsilon, K)$-component $\beta$ of $\gamma_i$ such that $z\in \beta$. Consider the tight subpath $[\beta_+, y_i]_{\gamma_i}$ and its $K$-truncation $\beta_1$. By the argument in Lemma \ref{tightadmissible}, the path $[z, \beta_+]\cdot \beta_1$ is a $c_2$-quasi-geodesic for some $c_2>0$.

Since $X_i$ is quasi-convex, there exists $\varepsilon=\varepsilon(\epsilon, c_2)>0$ such that any $c_2$-quasi-geodesic with two endpoints in $N_\epsilon(X_i)$ lies in $N_\varepsilon(X_i)$. This implies that $\beta_1\subset N_\varepsilon(X_i)$. However,  there exists no $(\epsilon, K)$-components in $[\beta_+, y_i]_{\gamma_i}$. Indeed, if not, there exists an $(\epsilon, K)$-component $\beta'$ in  $[\beta_+, y_i]_{\gamma_i}$ and $Y \in \mathbb P$ such that $\beta'_\pm\in N_\epsilon(Y)$ and $d(\beta'_-,\beta'_+)>K>l_0$. Since $y_i$ is the entry point of $\gamma_i$ in $N_\epsilon(X_i)$, we have $Y\ne X_i$.
Since $\beta'_\pm\in N_\varepsilon(X_i)$, we get $d(\beta'_-,\beta'_+)\le R:=\mathcal R(\max\{\epsilon, \varepsilon\})$ by Lemma \ref{periphsystem}. This is a contradiction as $l_0>\mathcal R(\max\{\epsilon, \varepsilon\})$. The same reasoning shows that $d(z, \beta_+)\le R$.

Let $L=L(\varepsilon, 1)$ be given by Lemma \ref{uniformtrans}. If $d(z, y_i)>2L+R$, there exists an interior point  in $[\beta_+, y_i]_{\gamma_i}$ which is $(\epsilon_c, 1)$-deep in $X_i$. This is a contradiction, since $y_i$ is the entry point of $[\beta_+, y_i]_{\gamma_i}$ in $N_\epsilon(X_i)$ and $\epsilon\ge \epsilon_c$ by Convention \ref{epsilonR1}. Hence we proved that $d(z, y_i)\le C:=2L+R$. The claim is proved.
\end{proof}

By the contracting property \ref{periphsystem} we see that  $\proj_{X_i}(\bar \gamma_i)  \le \tau:=2(D_{c_1}+\epsilon)+C$. The same is true for  $\proj_{X_i}(\bar \gamma_{i+1})$. Then  $\tilde \gamma$ satisfies the following properties:
\begin{enumerate}
\item
Each $\bar\gamma_i$ is a $c_1$-quasi-geodesic,
\item
$\max\{\proj_{X_i}(\bar \gamma_i), \proj_{X_i}(\bar \gamma_{i+1})\}  \le \tau$,
\item
$d(x_i, y_{i+1}) >l$.
\end{enumerate}
(These properties imply that $\tilde \gamma$ is $(l, c_1, c_1, \tau)$-admissible   in the sense of \cite[Section 3]{YANG6}). Therefore by Corollary 3.3 in \cite{YANG6}, there exist $l_0, c'>0$ such that for any $l>l_0$ the truncation $\tilde \gamma$ is a $c'$-quasi-geodesic. The lemma is proved.
\end{proof}
\begin{rem}
An alternative way to prove the above Lemma is to use the arguments of Proposition 6.1.1 in \cite{GePo4} to prove that $\tilde \gamma$ is a curve whose distortion is a quadratic polynomial,  then it follows from Proposition 7.2.2 in \cite{GePo4} that $\tilde \gamma$ is linearly distorted.
\end{rem}

\begin{prop}[Approximation by generalized tight paths] \label{shortcutpaths}
For any $l\ge 0$, there exists $0<\lambda_0<1$ such that the following property holds for any $\lambda\in [\lambda_0, 1)$.

For any $\xi \ne \eta \in \pG$, there exists a sequence of generalized $(1, l)$-tight paths $\gamma_n$ with $(\gamma_n)_- \in [o, \xi]$,
$(\gamma_n)_+ \in [o, \eta]$ such that $$\lim_{n\to \infty} d(o, (\gamma_n)_-) = \lim_{n\to \infty} d(o, (\gamma_n)_+) = \infty$$ and $$\lim_{n\to \infty} \flo(\gamma_n) = \sfdo(\xi, \eta).$$
\end{prop}
\begin{proof}
By definition  of the shortcut metric (\ref{shortcutdefn}), for any $\varepsilon>0$, there are
finitely many pairs $(\eta_i, \xi_{i+1}) \in \omega$ where $1\le i < m$ such that
\begin{equation}\label{sumtightpaths}
\sfdo(\xi, \eta) \ge \sum_{1\le i \le m} \fdo (\xi_i, \eta_{i}) -\varepsilon/3,
\end{equation}
where $\xi_1:=\xi, \eta_{m}:=\eta$. If $m=1$, the proof is completed by Lemma \ref{tithapprox}. Assume that $m\ge 2$.

Let $\kappa=\min\{\fds_o(\eta_i, \xi_{i+1}): 1\le i < m\}>0$. For each $1\le i< m$, there exists $X_i\in \mathbb P$ such that $\eta_i, \xi_{i+1} \in \partial_\lambda(X_i)$ where $\partial_\lambda(X_i)$ is the topological boundary of $X_i$ in $\pGf$.

First we claim that one can choose $\tilde \xi_1, \tilde \eta_{m}$ and $\tilde \eta_i, \tilde \xi_{i+1} \in X_i$ for each $1\le i< m$ such that the following two conditions hold,
\begin{enumerate}
\item
$\max\{\fds_o(\tilde \xi_i, \xi_i), \fds_o(\tilde \eta_i, \eta_i)\} \le \min\{\kappa/4, \frac{\varepsilon}{6m}\}$ for $1\le i\le m$.
\item
If there exists a path $\alpha$ between $\tilde \eta_i, \tilde \xi_{i+1}$ for $1\le i< m$ such that $\len(\alpha)\le 3l$, then it has $\fls_o$-length at most $\kappa/4$.

\end{enumerate}
Indeed, (1) is true for $\tilde \xi_i$ and $\tilde \eta_i$ sufficiently close to $\xi_i$ and $\eta_i$
respectively.  To prove (2), let $R=\min\{d(1, \tilde \xi_i), d(1,\tilde\eta_i): 1\le i\le m\}.$  We have $d(1, \alpha)\geq R-3l$. So for sufficiently large  $R$ the statement (2) follows from the visibility lemma \ref{karlssonlem}.

By Lemma \ref{tithapprox}, we can connect
$\tilde \xi_{i}, \tilde \eta_{i}$ by a $(1, l)$-tight path $\gamma_i$ for $1\le i\le m$ such that $(\gamma_i)_-=\tilde \xi_{i}$ and $(\gamma_i)_+=\tilde \eta_{i}$ and
\begin{equation}\label{approxbytightpaths}
|\fdo (\tilde \xi_i, \tilde \eta_{i})-  \flo (\gamma_i)|\le \displaystyle{\varepsilon\over 6m}.
\end{equation}
By the condition (1) above,  (\ref{sumtightpaths}) and (\ref{approxbytightpaths}) , the following holds
\begin{equation}\label{sumoftightpaths}
\sfdo(\xi, \eta) \ge  \sum_{1\le i \le m} \flo (\gamma_i) -{5\varepsilon \over 6}.
\end{equation}

Let $y_{i}, x_{i+1}$ be the entry and exit points of $\gamma_{i}$ and $\gamma_{i+1}$ in $N_\epsilon(X_i)$ respectively. If $d(y_i, x_{i+1}) \ge  l$ for all $1\le i < m$, then we are done: $\{\gamma_i\}$ give the generalized tight path. Otherwise, assume that $d(x_{j+1}, y_{j}) \le l$ for some $1\le j < m$.

Observe that $\max\{d(\tilde \eta_{j}, y_j),  d(\tilde \xi_{j+1}, x_{j+1})\}\ge l+1$. Indeed, if not, it follows that $\tilde \eta_j, \tilde \xi_{j+1}$ are connected by a path of length at most $3l$. By the above condition (2), we have $\fds_o(\tilde \eta_j, \tilde \xi_{j+1}) \le \kappa/4$. By the condition (1), we have $\fds_o(\eta_j, \xi_{j+1}) \le 3\kappa/4$.  We arrive at a contradiction with the definition of $\kappa$. Thus, we proved that $\max\{d(\tilde \eta_{j}, y_j),  d(\tilde \xi_{j+1}, x_{j+1})\}\ge l+1$. By Lemma \ref{Floydgeod}, we obtain the following
$$
\flo([y_j, \tilde \eta_j]_{\gamma_j}) + \flo([\tilde \xi_{j+1}, x_{j+1}]_{\gamma_{j+1}})  \ge \flo([y_j, x_{j+1}]).
$$
which yields
$$
\begin{array}{lll}
\flo(\gamma_j)+\flo(\gamma_{j+1}) &\ge \flo([\tilde \xi_{j}, y_j]_{\gamma_j}) +\flo([y_j, x_{j+1}]) +  \flo([x_{j+1}, \tilde \eta_{j+1}]_{\gamma_{j+1}})\\
&\ge \fdo(\tilde \xi_{j}, \tilde \eta_{j+1}).
\end{array}
$$
This implies that we can drop the pair $(\eta_j, \xi_{j+1})$ in (\ref{sumtightpaths}) such that the corresponding inequality in (\ref{sumoftightpaths}) still holds. Precisely, choose a $(1, l)$-tight path $\alpha_{j}$ between $\tilde \xi_j, \tilde \eta_{j+1}$ such that $$|\flo(\alpha_{j})-\fdo (\tilde \xi_j, \tilde \eta_{j+1})|\le {\varepsilon \over 6m}.$$ So $\flo(\gamma_j)+\flo(\gamma_{j+1}) \ge \flo(\alpha_{j})-\displaystyle{\varepsilon \over 6m}. $
It follows by (\ref{sumoftightpaths}), $$\sfdo(\xi, \eta) \ge \sum_{1\le i \le m; i \ne j, j+1} \flo (\gamma_i) +\flo(\alpha_{j}) - {5\varepsilon\over 6} - {\varepsilon \over 6m}.$$

Consider the new set of $(1, c)$-tight paths $\gamma_i$ ($ i \ne j, j+1$) and $\alpha_{j}$. Repeat the above argument for those $j$ for which $d(x_{j+1}, y_{j}) \le l$.  Since $m$ is finite, for every $\varepsilon>0$ we obtain a generalized tight path $\gamma$ such that $\sfdo(\xi, \eta) \ge \flo(\gamma)-\epsilon$. The Proposition is proved.
\end{proof}

\subsection{Floyd and shortcut metrics on uniformly conical points}
A priori, the shortcut metrics as quotient of the Floyd metrics might be distorted in a unexpected way. The main result of this subsection is to show that this distortion is not severe for uniformly conical points.

Fix a basepoint $o\in G$.  Recall that, in Section \ref{Section2.4},  $\uGLo$ denotes the set of uniformly conical points  $\xi\in \pG$ for which there exists an $(\epsilon, R, L)$-transitional geodesic ray between $o$ and $\xi$. Similarly, denote by $\uGfLo$ the set of uniformly conical points in $\pGf$ \textit{based} at $o$.  By Proposition \ref{FloydKernel}, there exists one-to-one correspondence between $\uGLo$ and $\uGfLo$.


The following is a version of Proposition \ref{tightbetweenconical} for generalized tight paths.
\begin{prop}\label{gromovproduct}
There exist $l_0, D>0$ such that for any $L>0$, there exists $M=M(L)>0$ with the following property.

Denote $\alpha_1=[o, \xi]$ and $\alpha_2=[o, \eta]$ for $\xi\ne \eta \in {\mathcal C}_LG$. Let $\gamma$ be a generalized $(1, l)$-tight path for some $l\ge l_0$ with $\gamma_-\in \alpha_1$ and $\gamma_+\in \alpha_2$. If $d(o, \gamma_-), d(o, \gamma_+)\gg 0$, then there exists $z \in \gamma$ such that $d(z, \alpha_1\cup \alpha_2)\le D$ and $d(z, \alpha_i) \le M$ for $i=1, 2$. Moreover, $|d(o, z)-d(o, [\xi, \eta])|\le M$.
\end{prop}

\begin{proof}
Let $l_0, c'\ge c, K>0$ given by Lemma \ref{jointpath} such that the $K$-truncation $\tilde \gamma$ of a generalized $(c, l)$-tight $\gamma$ for $l\ge l_0$  is a $c'$-quasi-geodesic. By Proposition \ref{tightbetweenconical}, there exists an $(\epsilon_{c'}, R_{c'})$-transitional point $z$ in $\tilde \gamma$ such that the conclusion of this Proposition holds. If the point $z$ lies on $\gamma$, then we are done. So below, we assume that $z \notin \gamma$, and then have two cases to consider as follows.

Let $L_1=L(\epsilon_{c}, R_{c'})$ be given by Lemma \ref{uniformtrans}, where $\epsilon_{c'}\ge \epsilon_{c} \ge \epsilon_{1}$ and they all satisfy Convention \ref{epsilonR1}.  Assume that $l_0 \ge 2L_1$.

\textbf{Case 1.}
The point $z$ lies in some $(\epsilon_{c}, K)$-component $\beta$ of a $(c, l)$-tight path $\gamma_i$.  Then $\max\{d( \beta_-, \gamma), d( \beta_+, \gamma)\} \le L_1$.
Indeed, if not,  then $\min\{d(z, \beta_-), d(z, \beta_+)\} \ge L_1$.  By applying Lemma \ref{uniformtrans} for the geodesic $\beta$, we have $z \in \tilde \gamma$ is $(\epsilon_{1}, R_{c'})$-deep in $X_j$. This is a contradiction, as $z$ is an $(\epsilon_{c'}, R_{c'})$-transitional point in $\tilde \gamma$.


\textbf{Case 2.}  The point $z$ lies in some $[y_j, x_{j+1}]$ for some $j$, where $[y_j, x_{j+1}]$ is given in Definition \ref{gentightdefn} of a generalized tight path.
By the same reasoning as above, we apply Lemma \ref{uniformtrans} for the geodesic $[ (\beta)_-, (\beta)_+]$. Then $\min\{d(z, y_j), d(z, x_{j+1})\} \le L_1$.

Thus, we proved that $z$ has a distance at most $2L_1$ to a point in $\gamma$. The conclusion follows as a consequence of Proposition \ref{tightbetweenconical}.
\end{proof}

The main result of this subsection   is the following.

\begin{prop}[Visual Floyd/shortcut metric]\label{FloydMetric}
There exists $0< \lambda_0<1$ such that the following holds for any $L>0$ and $\lambda\in [\lambda_0, 1)$.

We have $$\fdo(\xi,
\eta) \asymp_{L} \lambda^{n},\; \forall \;\xi \ne \eta \in \uGfLo
$$
and
$$\sfdo(\xi,
\eta) \asymp_{L} \lambda^{n},\; \forall \;\xi \ne \eta \in \uGLo $$ where $n=d(o, [\xi,\eta])$.
\end{prop}

\begin{proof}
Let us consider the shortcut metric case only. The Floyd metric case is similar and even easier.

Let $\alpha_1, \alpha_2$ be two $(\epsilon_1, R_1, L)$-transitional geodesic rays
originating at $o$ and terminating at $\xi, \eta$
respectively.

Let $l_0, D>0$ be given by Proposition \ref{gromovproduct}, and we choose $\lambda_0 \in ]0, 1[$ verifying Proposition \ref{shortcutpaths} for $l=l_0$.  Then by Propositions   \ref{gromovproduct} and \ref{shortcutpaths} there exists $M=M(L)>0$ such that the following holds:
\begin{enumerate}
\item
For each $k>0$, there exists a sequence of generalized $(1, l_0)$-tight paths $\gamma_k$ with $(\gamma_k)_- \in \alpha_1, (\gamma_k)_+ \in \alpha_2$ and such that $(\gamma_k)_-  \to \xi$,
$(\gamma_k)_+  \to \eta$ and \begin{equation}\label{totalength2}
|\flo(\gamma_k) -\sfdo(\xi, \eta)|\le  1/k.
\end{equation}
\item
There exists $z_k \in \gamma_k$ such that $d(z_k, \alpha_1\cup \alpha_2)\le D$ and $d(z_k, \alpha_i) \le M$ for $i=1, 2$. Moreover, $|d(o, z_k)-d(o, [\xi, \eta])|\le M$.
\end{enumerate}
Denote $u_k:=(\gamma_k)_-$ and $v_k:=(\gamma_k)_+$.

\textbf{Upper bound.} Choose $x_k\in \alpha_1$ and $y_k\in \alpha_2$ such that $\max\{d(z_k, x_k), d(z_k, y_k)\}\le M$. Then for every point $t\in [x_k, z_k]\cup[z_k,y_k]$ we have
$$d(o,t) \ge d(o, z_k) -M\ge n-2M.$$ Hence
$$\max\{\flo([x_k, z_k]), \flo([y_k,z_k])\}\leq M\cdot \lambda^{n-2M}.$$
We also have
$$
\begin{array}{lll}
\max\{\flo([x_k, \xi]_{\alpha_1}), \flo([y_k, \eta]_{\alpha_2})\} \le \frac{\displaystyle{\lambda^{\min\{d(o, x_k), d(o, y_k)\}}}}{1-\lambda}  \le \frac{\lambda^{n-2M}}{1-\lambda}.
\end{array}
$$
It follows that
\begin{equation}\label{halflength1}
\begin{array}{lll}
\sfdo(\xi, \eta)
&\le \flo([x_k,\xi]_{\alpha_1}) + \flo([y_k, \eta]_{\alpha_2}) + \flo([x_k, z_k])+\flo([y_k, z_k])   \\
\\
&\le 2 \lambda^{n-2M} (\frac{1}{1-\lambda} +
M).
\end{array}
\end{equation}

Let $C_1:=2 \lambda^{-2M} (\frac{1}{1-\lambda} +
M)$. Then $\sfdo(\xi, \eta) \le  C_1 \lambda^{n}.$

\textbf{Lower bound.} Since $d(z_k, \alpha_1\cup \alpha_2)\le D$, there exists $w_k \in \alpha_1\cup \alpha_2$ such that $d(z_k, w_k)\le D$. Assume that $w_k\in \alpha_2$ for concreteness.

By Lemma \ref{floydrays}, any segment of $\alpha_2$ is a Floyd geodesic. Since $v_k \to \eta$ and $d(o, w_k)\le n+D$, we can assume that $w_k\in [o, w_k]_{\alpha_2}$ for all $k\gg 0$. So
$$
\begin{array}{lll}
\fdo(v_k, w_k)=\flo([w_k, v_k]_{\alpha_2}) &= \frac{\lambda^{d(o, w_k)}}{1-\lambda}-\frac{\lambda^{d(o, v_k)}}{1-\lambda} \\
\\
& \ge \frac{\lambda^{d(o, z_k)+D}}{1-\lambda} -\frac{\lambda^{d(o, v_k)}}{1-\lambda} .
\end{array}
$$
We have $$\fdo(v_k, w_k)  \le \flo(\gamma_k) +\flo([w_k, z_k]).$$
Since $d(z_k, w_k)\le D$ we have
$$
\flo([w_k, z_k]) \le D\cdot \lambda^{d(o, z_k)-D}.
$$
Thus
\begin{equation}\label{halflength2}
\begin{array}{lll}
\flo(\gamma_k) &\ge \fdo(v_k, w_k) -\flo([w_k, z_k])  \\
\\
&\ge \flo([v_k, w_k]_{\alpha_2}) -\flo([w_k, z_k])  \\
\\
&\ge (\frac{\lambda^{D}}{1-\lambda} -
\frac{D}{\lambda^{D}})\lambda^{d(o, z_k)}-\frac{\lambda^{d(o, v_k)}}{1-\lambda}\\
\\
&\ge (\frac{\lambda^{D}}{1-\lambda} -
\frac{D}{\lambda^{D}})\cdot\lambda^M\cdot \lambda^{n}-\frac{\lambda^{d(o, v_k)}}{1-\lambda}.
\end{array}
\end{equation}
Since $D$ does not depend on $L$ by Lemma \ref{gromovproduct},  there exists   $1>\lambda_0>0$ such that
\begin{equation}\label{crucialinquality}
\frac{\lambda^{D}}{1-\lambda} -
\frac{D}{\lambda^{D}}\ge \frac{\lambda_0^{D}}{1-\lambda_0} -
\frac{D}{\lambda_0^{D}}>0,
\end{equation}
for any $\lambda \in [\lambda_0, 1)$.
Let $C_2:=(\frac{\lambda_0^{D}}{1-\lambda_0} -
\frac{D}{\lambda_0^{D}}) \cdot\lambda^M>0$. Note that $d(o, v_k) \to \infty$ as $k\to \infty$.   By (\ref{totalength2}) and (\ref{halflength2}),  passing to the limit when $k\to \infty$,  we obtain
$$
\sfdo(\xi, \eta) \ge  C_2 \lambda^{n},
$$
for any $\xi\ne \eta \in  \uGLo$ and any $L>0$. The proof is then complete.
\end{proof}
\begin{rem}
\begin{enumerate}
\item
The fact that the constant $D$ does not depend on $L$ is crucial for the choice of $\lambda_0$ in (\ref{crucialinquality}).
\item
This lemma gives an asymptotic formula for two uniformly conical points with respect to Floyd metric and shortcut metric.  This could be used to give an alternative proof of Lemma \ref{ShadowDisk}, but cannot be derived from (the proof of) Lemma \ref{ShadowDisk}.
\end{enumerate}
\end{rem}

\section{Appendix: Visual metrics and Floyd metrics are bilipschitz equivalent}
The aim of the Appendix is to give a short proof that  the visual Gromov metric $\nu_{a, o}$ and the Floyd metric $\fdo$ on the boundary $\partial X$ of   a $\delta$-hyperbolic graph $(X, d)$ are bilipschitz equivalent for some choice of parameters $a$ and $\lambda.$ This fact, mentioned in the Introduction, is often considered as a folklore, however we have not found a complete proof of it in the literature   (see e.g. \cite[Lemma 7.2.1]{Gro} and the key inequality after it, or \cite[Formula (1.3)]{Coor},  in both cases the fact is stated without proof).

Recall the definition of the Gromov visual metric $\nu$ on $\partial X$. For a real parameter $a>0$ set $\delta_{a,o}(\xi, \eta)=e^{-a(\xi\vert\eta)},$ where $(\xi\vert\eta)$ denotes  the Gromov product for the basepoint $o.$
Let \begin{equation}\label{defvis}\nu_{a, o}(\xi,\eta)=\inf\Big\{\sum_{i=0}^n\delta_{a,o}(c_{i-1}, c_i)\ :\ c_i\in{\mathcal C}_{\xi,\eta}\Big\},\end{equation} where ${\mathcal C}_{\xi,\eta}$ is the set of chains of points in $\partial X$ such that $c_0=\xi$ and $c_n=\eta.$ If $\displaystyle a < {\ln 2\over 6\delta}$  then $\nu_{a, o}$ is a metric on $\partial X$   satisfying the following  inequality \cite[Proposition 7.10]{GH}:
\begin{equation}\label{productmetric}
(3-2e^{a\delta})\cdot \delta_{a,o}(\xi, \eta)\leq \nu_{a, o}(\xi,\eta)\leq \delta_{a,o}(\xi, \eta)
\end{equation}

 \begin{rem}We note that an inequality similar to (\ref{productmetric}) where the metric $\nu$  is replaced by the Floyd metric (which is our goal now) is formally stated in  \cite[page 5]{Pu}    but no  justification is given.
\end{rem}
For  the Gromov product the following inequality is true \cite[Lemma 2.7]{CDP}:

 \begin{equation}\label{productgeodist}
e^{-4a\delta}\cdot \delta_{a,o}(\xi, \eta)\leq e^{-a d(o,[\xi, \eta])}\leq  \delta_{a,o}(\xi, \eta),
\end{equation}

\noindent where $d(o,[\xi, \eta])$ is the distance in $X$ from $o$ to the union of all geodesics between $\xi$ and $\eta.$ The inequalities  (\ref{productmetric}) and  (\ref{productgeodist}) imply:

\begin{equation}\label{visualmetricP}
\nu_{a, o}(\xi, \eta)\asymp _{C_1}e^{-a\cdot d(o, [\xi, \eta])}
\end{equation}
  for the constant $C_1=\max\{e^{4a\delta}, 3-2e^{a\delta}\}.$

The following proposition provides the bilipschitz equivalence between the  visual metric and the Floyd metric on the boundary of   $X.$

\begin{prop}\label{VisualFloyd}
Let $(X, d)$ be a $\delta$-hyperbolic graph. There exist a constant  $a_0$ such that for any $a\in ]0, a_0]$  there exists a constant $C$ for which
\begin{equation}\label{bilipsch}
\forall \xi, \eta \in \pX: \nu_{a, o}(\xi, \eta) \asymp_C \fdo(\xi, \eta), 
\end{equation}
\noindent where  $\lambda=e^{-a}.$
\end{prop}
\begin{proof}
By (\ref{defvis}) and (\ref{visualmetricP}) we have

$$\displaystyle \nu_{a,o}(\xi, \eta)=\inf_{{\mathcal C}_{\xi,\eta}}\sum_{i=0}^n \Big\{\lambda^{d(o, [c_{i-1}, c_i])}\ :\ c_i\in {\mathcal C}_{\xi, \eta}\Big\}\leq C_1\rho_{\lambda,o}(\xi, \eta).$$

\noindent Indeed $\rho_{\lambda,o}$ is obtained by taking  infimum of the  expression  (\ref{defvis}) over the subset of $\mathcal C_{\xi,\eta}$ given by  the sets of vertices of paths between $\xi$ and  $\eta$ (see Section 2.2).

To prove the opposite inequality, we need to use the $\delta$-thin triangle property. Consider a geodesic triangle with vertices $o, x, y$ in $X.$ There exists a $\delta$-center $c$ on $[x, y]$ such that $d(c, [o,x]) \le \delta$ and $d(c, [o,y])\le \delta$. For notational simplicity we ignore a small  uniform difference between   different hyperbolicity constants   (see e.g. \cite[Proposition 2.21]{GH}), and denote all of them by $\delta>0$. Since  $|d(o, c)-d(o, [x,y])|$ is uniformly upper bounded to simplify the notations again we  assume that  $d(o,c)=d(o, [x,y]).$

Choose $x'\in [o, x]$ and $y'\in [o, y]$ so that $\max\{d(c, x'), d(c,y')\}\le \delta$. We have $ \min\{d(o, [x', c]), d(o, [y',c]) \geq d(o, [x,y])-\delta$. Hence the Floyd length $\flo([x',c])$ of $[x',c]$ is at most $$\delta\cdot \lambda^{d(o,[x',c])}\geq \delta\cdot \lambda^{d(o, [x,y])-\delta},$$ and similarly for $\flo([y',c]).$  Since   $[x', x]$ and $[y', y]$ are Floyd $\fdo$-geodesics (Lemma \ref{floydrays}) we have $$\displaystyle\flo([x', x])={\lambda^{d(o, x')}\over 1-\lambda}\leq {\lambda^{d(o,[x,y])-\delta}\over 1-\lambda},$$   and the same for     $\flo([y', y])$. Summing all up we obtain the following estimation for the Floyd length of $[x,y]$:
$$
\begin{array}{ll}
\fdo(x, y)&\le \flo([x', x])+\flo([x', c])+\flo([y', c])+\flo([y', y])\\
&\le 2\frac{\lambda^{d(o, [x,y])}}{1-\lambda}+   2\delta\cdot \lambda^{d(o, [x,y])-\delta}\\
&\le C_2\cdot \lambda^{d(o, [x,y])},
\end{array}
$$
for the constant   $\displaystyle C_2=\max\Big\{{2\over 1-\lambda}, {2\delta\over \lambda^{\delta}}\Big\}$.  Passing to the limits when $x\to \xi\in\partial X$ and $y\to\eta\in\partial X$, and using (\ref{visualmetricP}) we obtain (\ref{bilipsch})  for the constant $C=\max\{C_1, C_2\}.$

\end{proof}

\bibliographystyle{amsplain}
 \bibliography{bibliography}

\end{document}